\documentclass[11pt,reqno]{amsart}
\usepackage{amsmath,amssymb,amscd,graphicx, amsthm, %fullpage,
  enumerate, hyperref} 
\usepackage{mathrsfs}
\usepackage{bbm}
\usepackage{mathtools}
\usepackage[all]{xy}
\usepackage{url}
\usepackage[margin=2cm]{geometry}
\usepackage{color}
\definecolor{darkred}{rgb}{1,0,0} % change the intensity in [0,1]
\definecolor{darkgreen}{rgb}{0,0.8,0}
\definecolor{darkblue}{rgb}{0,0,1}

\hypersetup{colorlinks,
linkcolor=darkblue,
filecolor=darkgreen,
urlcolor=darkred,
citecolor=darkgreen}

\numberwithin{equation}{section}

\theoremstyle{definition}
\newtheorem{theorem}{Theorem}
\numberwithin{theorem}{section}

\newtheorem{lemma}[theorem]{Lemma}

\newtheorem{corollary}[theorem]{Corollary}

\newtheorem{definition}[theorem]{Definition}
\newtheorem{remark}[theorem]{Remark}

\newtheorem{notation}[theorem]{Notation}

\theoremstyle{remark}

\newcommand{\inv}{^{-1}} 

\newcommand{\toto}{\rightrightarrows}

\newcommand{\id}{\mathrm{id}}
\newcommand{\bone}{\mathbf{1}}

\newcommand{\ba}{\mathbf{a}}
\newcommand{\ha}{\hat{\ba}}
\newcommand{\GL}{\mathrm{GL}}

\newcommand{\Aut}{\mathrm{Aut}}

\newcommand{\Hom}{\mathsf{Hom}}
\newcommand{\sfC}{\mathsf{C}}
\newcommand{\sfA}{\mathsf{A}}
\newcommand{\sfB}{\mathsf{B}}
\newcommand{\sfD}{\mathsf{D}}
\newcommand{\Man}{\mathsf{{Man}}}
\newcommand{\Lie}{\mathsf{{Lie}}}

\newcommand{\LietAlg}{\mathsf{Lie2Alg}}
\newcommand{\LieGp}{\mathsf{LieGp}}
\newcommand{\LieGpd}{\mathsf{LieGpd}}

\newcommand{\Vect}{\mathsf{Vect}}
\newcommand{\cL}{{\mathcal L}}
\newcommand{\calX}{{\mathcal X}}

\newcommand{\scrL}{{\mathscr L}}
\newcommand{\fg}{\mathfrak g}

\newcommand{\fh}{\mathfrak h}

\newcommand{\X}{\mathbb{X}}
\newcommand{\R}{\mathbb{R}}

\newcommand{\bbm}[1]{\mathbbm{#1}}

\usepackage{longtable}

\title{Left-invariant vector fields on a Lie 2-group}

\author{Eugene Lerman }

\address{Department of Mathematics, University of Illinois, Urbana,
  IL 61801}

\begin{document}

\begin{abstract}
  A Lie 2-group $G$ is a category internal to the category of Lie
  groups.  Consequently it is a monoidal category and a Lie groupoid.
  The Lie groupoid structure on $G$ gives rise to the Lie 2-algebra
  $\X(G)$ of multiplicative vector fields, see
  \cite{B-E_L}.  The monoidal structure on $G$ gives rise to a left
  action of the 2-group $G$ on the Lie groupoid $G$, hence to an
  action of $G$ on the Lie 2-algebra $\X(G)$.  As a result we get the
  Lie 2-algebra $\X(G)^G$ of left-invariant multiplicative vector
  fields.   

  On the other hand there is a well-known construction that associates
  a Lie 2-algebra $\fg$ to a Lie 2-group $G$: apply the functor $\Lie:
  \mathsf{LieGp} \to \mathsf{LieAlg}$ to the structure maps
  of the category $G$.  
  We show that the Lie 2-algebra $\fg$ is isomorphic to the Lie
  2-algebra $\X(G)^G$ of left invariant multiplicative vector fields.
\end{abstract}

\maketitle
\setcounter{tocdepth}{1}
\tableofcontents

%%%%%%%%%%%%%%%%%%%%%%%%%%%%%
\section{Introduction}  \label{sec:1}
%%%%%%%%%%%%%%%%%%%%%%%%%%%%%

% Recall that given a category $\sfC$ with finite limits we can talk
% about  a {\em category $C$ internal to the category $\sfC$}.  Such $C$
% would consist of two objects $C_0$, $C_1$ of $\sfC$ (object and
% morphisms, respectively) and 

Recall that a strict Lie 2-group $G$ is a category internal to the
category $\LieGp$ of Lie groups (the notions of internal categories,
functors and natural transformations are recalled in
Definition~\ref{def:internal}).  Thus $G$ is a category whose collection of
objects is a Lie group $G_0$, the collection of morphisms is a Lie
group $G_1$ and all the structure maps: source $s$, target $t$, unit
$1:G_0 \to G_1$ and composition $*:G_1\times_{s,G_0,t} G_1 \to G_1$
are maps of Lie groups.

There is a well-known functor $\Lie:\mathsf{LieGp} \to \mathsf{LieAlg}$ from the
category of Lie groups to the category of Lie algebras.  The functor
$\Lie$ assigns to a Lie group $H$ its tangent space at the identity
$\mathfrak{h}=T_eH$. The Lie bracket on $\fh$ is defined by the
identification of $T_eH$ with the Lie algebra of left-invariant vector
fields on the Lie group $H$.  To a map $f:H\to L$ of Lie groups the
functor $\Lie$ assigns the differential $T_ef: T_eH \to T_e L$, which
happens to be a Lie algebra map.  Consequently given a Lie 2-group
$G=\{G_1 \toto G_0\}$ we can apply the functor $\Lie$ to all the
structure maps of $G$ and obtain a (strict) Lie 2-algebra $\fg=
\{\fg_1\toto \fg_0\}$.

On the other hand, any Lie 2-group happens to be a Lie groupoid.  In fact, it is
an action groupoid \cite[Proposition~32]{BL} (see also
Corollary~\ref{cor:2.7} below). Hepworth in
\cite{Hepworth} pointed out that any Lie groupoid $K$ possesses a
{\em category} $\X(K)$ of vector fields (and not just a vector space).  The objects of this category are
well-known multiplicative vector fields of Mackenzie and Xu
\cite{MackXu}. Multiplicative vector fields on a Lie groupoid
naturally form a Lie algebra.  It was shown in \cite{B-E_L} that the
space of morphisms of $\X(K)$ is a Lie algebra as well, and moreover
$\X(K)$ is a strict Lie 2-algebra (that is, a category internal to Lie
algebras).  One may expect that for a Lie 2-group $G$ one can define
the Lie 2-algebra $\X(G)^G$ of left-invariant vector fields on $G$ and
that this Lie 2-algebra is isomorphic to the Lie 2-algebra $\fg$.  But
what does it mean for a Lie 2-group to act on its Lie 2-algebra?  And what
does it mean to be left-invariant for such an action?  We proceed by
analogy
with ordinary Lie groups.\\

A Lie group $H$ acts on itself by left multiplication: for any $x\in
H$ we have a diffeomorphism $L_x:H\to H$, $L_x (a) := xa$.  These
diffeomorphisms, in turn, give rise to a representation
\[
\lambda:H\to \mathrm{GL}(\calX (H))
\]
of the group $H$ on the vector space $\calX (H)$ of vector fields on
the Lie group $H$.  Namely, for each $x\in H$, the linear map
$\lambda(x): \calX (H) \to \calX(H)$ is defined by
\[
\lambda (x)v:= TL_x \circ v \circ L_{x\inv}
\]
for all vector fields $v\in \calX(H)$.  
Next recall that given a representation $\rho: H\to \mathrm{GL}(V)$ of
a Lie group $H$ on a vector space $V$ the space $V^H$ of $H$-fixed
vectors is usually defined by
\[
V^H =\{v\in V\mid \rho(x) v = v \textrm{ for all } x\in H\}.
\]
The space $V^H$ has the following universal property: for any linear
map $f:W\to V$ (where $W$ is some vector space) so that
\[
\rho(x) \circ f = f
\]
for all $x\in H$, there is a unique linear map $\bar{f}:W\to V^H$ so
that the diagram
\[
\xy
(0,12)*+{W}="1";
(-12, -5)*+{V^H}="2";
( 12, -5)*+{V}="3";
{\ar@{->}_{\bar{f}} "1"; "2"};
{\ar@{->}^{f} "1"; "3"};
{\ar@{->}^{\imath} "2"; "3"};
\endxy
\]
commutes.  Here $\imath:V^H \hookrightarrow V$ is the inclusion map.
If we view the group $H$ as a category $BH$ with one object $*$ and
$\Hom_{BH}(*,*) = H$, then the representation $\rho:H\to
\mathrm{GL}(V)$ can be viewed as the functor $\rho:BH\to \Vect$ (where
$\Vect$ is the category of vector spaces and linear maps) with
$\rho(*) =V$.  From this point of view the vector space $V^H$ of
$H$-fixed vectors ``is'' the limit of the functor $\rho$:
\[
V^H = \lim(\rho:BH\to \Vect).
\]
Consequently the vector space $\calX(H)^H$ of left-invariant vector
fields on a Lie group $H$ is the limit of the functor $\lambda: BH\to
\Vect$ with $\lambda(*) = \calX(H)$ and $\lambda (x) v = TL_x \circ v
\circ L_{x\inv}$ for all $x\in H, v\in \calX(H)$:
\[
\calX(H)^H = \lim (\lambda: BH\to \Vect).
\]
\mbox{}\\

Now consider a Lie 2-group $G$. Each object $x$ of $G$
gives rise to a functor $L_x:G\to G$ which is given on an arrow
$b\xleftarrow{\sigma}a$ of $G$ by
\[
L_x (b\xleftarrow{\sigma}a) = x\cdot b\xleftarrow{1_x\cdot \sigma} x\cdot a.  
\]
Here $\cdot $ denotes both multiplications:
in the group $G_0$ and in
the group $G_1$.  The symbol $1_x$ stands for the identity arrow at
the object $x$.  For any arrow $y\xleftarrow{\gamma} x$ of $G$ there is
a natural transformation
\[
L_\gamma:L_x \Rightarrow L_y.
\]
The component of $L_\gamma$ at an object $a$ of $G$ is defined by 
\[
L_\gamma (a) = y\cdot a \xleftarrow{\gamma \cdot 1_a} x \cdot a.
\]
  The proof that $L_\gamma$ is in fact a natural
transformation is not completely trivial; see Lemma~\ref{lem:2.2.0}.

Next recall that there is a tangent (2-)functor $T:\LieGpd\to \LieGpd$
from the category of Lie groupoids to itself.  This functor is an
extension of the tangent functor $T:\Man\to \Man$ on the category of
manifolds. On objects $T$ assigns to a Lie groupoid $K$ its tangent
groupoid $TK$. On morphisms $T$ assigns to a functor $f:K\to K'$ the
derivative $Tf:TK\to TK'$.  To a natural transformation
$\alpha:f\Rightarrow f'$ between two functors $f,f':K\to K'$ the
functor $T$ assigns the derivative $T\alpha$ (note that a natural transformation $\alpha$ is, in particular,  a
smooth map $\alpha:K_0 \to K'_1$, so $T\alpha: TK_0\to TK'_1$ makes
sense).  Note also that the projection functors $\pi_K:TK\to K$
assemble into a (2-)natural transformation $\pi:T\Rightarrow
\id_\LieGpd$.

Given an object $x$ of a Lie 2-group $G$ there is a functor $\lambda(x):\X(G)\to
\X(G)$ from the category of vector fields on the Lie groupoid $G$  to itself (see Lemma~\ref{lem:2.2.1} and the discussion right after it).  It is defined as follows: given a multiplicative vector field $v:G\to TG$,  the value of $\lambda(x)$ on $v$ is given by 
\[
\lambda(x) \, (v) :=  TL_x \circ v \circ L_{x\inv}.
\]
The value of $\lambda(x)$  on a morphism $\alpha:v\Rightarrow
w$ (i.e., on a natural transformation between the two functors) of $\X(G)$ is     the composite
\[
\xy
(-20,0)*+{TG}="1";
(-5,0)*+{TG}="2";
( 10,0)*+{G}="3";
(25,0)*+{G}="4";
{\ar@{->}_{TL_x} "2";"1"};
{\ar@{->}_{L_{x\inv}} "4";"3"};
{\ar@/^1.pc/^{w} "3";"2"};
{\ar@/_1.pc/_{v} "3";"2"};
{\ar@{=>}^<<<{\scriptstyle \alpha} (3,3)*{};(3,-3)*{}} 
\endxy.
\]
That is, 
\[
\lambda (x)\, (\alpha) :=TL_x\, \alpha\,  L_{x\inv},
\] 
the whiskering of the natural transformation $\alpha$ by the functors
$TL_x$ and $L_{x\inv}$.
Note that $\lambda(x)\circ \lambda (x\inv) = \id_{\X(G)}=
\lambda(x\inv)\circ \lambda(x)$.  And, more generally,
$\lambda(x)\circ \lambda(y) = \lambda(x\cdot y)$ for all objects $x,y$
of the Lie 2-group $G$.
For any arrow $y\xleftarrow{\gamma}x$ in the category $G$ we have a
natural transformation $\lambda(\gamma):\lambda(x)\Rightarrow
\lambda(y)$: its component 
\[
 \lambda(\gamma)v:\lambda(x)v\Rightarrow \lambda(y)v
\]
at a multiplicative vector field $v$  is given by the composite
\[
\xy
(-25,0)*+{TG}="1";
(-5,0)*+{\,\,TG}="2";
( 10,0)*+{G}="3";
(30,0)*+{G}="4";
{\ar@{->}_{v} "3";"2"};
{\ar@/^1.pc/^{TL_x} "2";"1"};
{\ar@/_1.pc/_{TL_y} "2";"1"};
{\ar@/^1.pc/^{L_{x\inv}} "4";"3"};
{\ar@/_1.pc/_{L_{y\inv}} "4";"3"};
{\ar@{=>}^<<<<{\scriptstyle TL_\gamma} (-15,3)*{};(-15,-3)*{}}; 
{\ar@{=>}^<<<<{\scriptstyle L_{\gamma\inv}} (19,3)*{};(19,-3)*{}}; 
\endxy.
\]
We can always think of a Lie 2-algebra $\X(G)$ as a 2-vector space
(i.e., a category internal to the category of vector spaces) by
forgetting the Lie brackets.  A 2-vector space has a strict 2-group of
automorphisms.  By definition the objects of this 2-group are strictly
invertible functors internal to the category of vector spaces and the
morphisms are natural isomorphisms (also internal to the category of
vector spaces). We denote the 2-group of automorphisms of $\X(G)$ by
$\GL(\X(G))$.  The functors $\lambda(x)$ and the natural
transformations $\lambda(\gamma)$ described above assemble into a
single homomorphism of 2-groups $\lambda: G\to \GL(\X(G))$ (i.e., a
functor internal to the category of groups), which we can think of as
the ``left regular representation" of the Lie 2-group $G$ on its
category of vector fields $\X(G)$.  The main result of the paper may
now be stated as follows.

\begin{theorem} \label{thm:main} Let $G$ be a (strict) Lie 2-group,
  $\fg$ its Lie 2-algebra obtained by applying the $\Lie$ functor to
  its structure maps, $\X(G)$ the Lie 2-algebra of multiplicative
  vector fields, and $\lambda: G\to \GL(\X(G))$ the representation of
  $G$ on the 2-vector space $\X(G)$ of multiplicative vector fields
  which arises from the left multiplication as described above.  There
  is a natural 1-morphism $p:\fg \to \X(G)$ of Lie 2-algebras which is
  fully faithful and injective on objects.  Hence the image $p(\fg)$
  of the functor $p$ is a full Lie 2-subalgebra of $\X(G)$.

  Moreover the inclusion $p(\fg)\lhook\joinrel\xrightarrow{\,\,\,
    i\,\,\,}\X(G)$ is the strict conical 2-limit of the functor
  $\lambda: G\to \mathrm{GL}(\X(G))$.  Hence the Lie 2-algebra $\fg$
  is isomorphic to the Lie 2-algebra $\X(G)^G:= \lim(\lambda :G\to
  \mathrm{GL}(\X(G)))$ of left-invariant vector fields on the Lie
  2-group $G$.
\end{theorem}

% \begin{remark}
% As we discussed above it is highly plausible that the one  should be able to realize the Lie 2-algebra of a Lie 2-group as a Lie 2-algebra of left-invariant vector field, where ``left-invariant" should mean a limit of the ``left regular representation" functor $\lambda: G\to \GL(\X(G))$.  The questions is --- what sort of a limit?   One of the surprising conclusion of Theorem~\ref{thm:main} is that the strict conical limit will do and no lax or weighted limits are required.

% %and in what 2-category?  Lie 2-algebras are objects in an ``obvious" 2-category $\mathsf{Lie2Alg}_{strict}$.   The 1-morphisms of $\mathsf{Lie2Alg}_{strict}$ are functors internal to the category of Lie algebras and 2-morphisms are the internal transformations.   There is another equally (if not more) natural category $\mathsf{Lie2Alg}$ obtained from $\mathsf{Lie2Alg}_{strict}$ by localizing it at essential equivalences.  Here ``essential equivalences" are by definition the internal functors that are fully faithful and essentially surjective as ordinary functors.     Unfortunately such essential equivalences need not even have weak inverses in  $\mathsf{Lie2Alg}_{strict}$.  The reason for this boils down to the fact that short exact sequences of Lie algebras need not split. 
% \end{remark}

\subsection*{Related work}   Higher Lie theory is a well-developed
subject.  The ideas go back to the work of Quillen \cite{Quillen} and
Sullivan \cite{Sullivan} on rational homotopy theory.   The problem of
associating a Lie 2-algebra to a strict Lie 2-group is, of course,
solved by applying a Lie functor to the Lie 2-group.   In fact a much
harder problem has been solved by \v Severa who introduced a {\sf
  Lie}-like functors  that go  from Lie $n$-groups  to $L_\infty$-algebras and
from  Lie $n$-groupoids to dg-manifolds \cite{S1,S2}.  In particular
one can use \v Severa's method to differentiate weak Lie 2-groups \cite{JSW}.

An even
harder problem is that of integration.  We note the work of Crainic
and Fernandes \cite{CF}, Getzler
\cite{Getzler}, Henriques \cite{Henriq} and  \v Severa and  \v
Sira\v n \cite{SS}.

\subsection*{Outline of the paper}\mbox{}\\
In Section~\ref{sec:2} we fix our notation, which unfortunately is
considerable.  We recall the definitions internal categories, of
2-groups, Lie 2-groups, Lie 2-algebras and 2-vector spaces.  We then
recall the interaction of composition and multiplications in a Lie
2-group and the fact that any Lie 2-group is a Lie groupoid.  We
discuss the category of vector fields $\X(K)$ on a Lie groupoid $K$
and the fact that this category is naturally a Lie 2-algebra.  In
particular we discuss the origin of the Lie bracket on the space of
morphisms of $\X(K)$.

In Section~\ref{sec:action} we discuss the 2-group of automorphisms of
a category.  We define an action of a 2-group on a category and
express the action in terms of a 1-morphism of 2-groups.  We show that
the multiplication of a Lie 2-group $G$ leads to an action $L:G\to
\Aut(G)$ of the group on itself by smooth (internal) functors and
natural isomorphisms.  We show that an action of a Lie 2-group $G$ on
a Lie groupoid $K$ by smooth (internal) functors and natural
isomorphisms leads to a representation of $G$ on the on the 2-vector
space $\X(K)$ of vector fields on $K$.  In particular left
multiplication $L:G\to \Aut (G)$ leads to a representation $\lambda:
G\to \mathrm{GL} (\X(G))$ of a Lie 2-group $G$ on its 2-vector space
of vector fields.  Various results of this section may well be known
to experts. I don't know of suitable references.

% $G \to {\mathrm GL}(\X(G)$  on the
% 2-vector space $\X(K)$ of vector fields on $K$.

In Section~\ref{sec:p} for a Lie 2-group $G$ we construct a 1-morphism
of Lie 2-algebras $p:\fg\to \X(G)$ which is fully faithful and
injective on objects.  Consequently the image $p(\fg)$ is a full Lie
2-subalgebra of the Lie 2-algebra of vector fields $\X(G)$.

Finally in Section~\ref{sec:u} we show that the inclusion $i:p(\fg)
\hookrightarrow \X(G)$ is a strict conical 2-limit of the left regular
representation $\lambda: G\to \mathrm{GL} (\X(G))$.

\subsection*{Acknowledgments} The paper is part of a joint project
with Dan Berwick-Evans.  I am grateful to Dan for many fruitful
discussions.

I thank the referee for the careful reading of the paper and for many
interesting and helpful comments.

\section{Background and notation}\label{sec:2}

\begin{notation}
  Given a natural transformation $\alpha:f\Rightarrow g$ between a
  pair of functors $f,g:\sfA\to \sfB$ we denote the component of
  $\alpha$ at an object $a$ of $\sfA$ either as $\alpha_a$ or as
  $\alpha(a)$, depending on readability.
\end{notation}

\begin{notation}
  Given a category $\sfC$ we denote its collection of objects by
  $\sfC_0$ and its collection of morphisms by $\sfC_1$. The
  source and target maps of the category $\sfC$ are denoted by $s,t:\sfC_1\to
  \sfC_0$, respectively.  The unit map from objects to morphisms is
  denoted by $1:\sfC_0\to \sfC_1$. We write
\[
\ast : \sfC_1 \times_{s,\sfC_0, t} \sfC_1 \to \sfC_1, \qquad (\sigma,
\gamma) \mapsto \sigma \ast \gamma
\]
to denote composition in the category $\sfC$.  Here and elsewhere
\[
\sfC_2:= \sfC_1 \times_{s,\sfC_0, t} \sfC_1 = \{(\gamma_2, \gamma_1)\in G_1
\times G_1 \mid s(\gamma_2) = t(\gamma_1)\}
\]
denotes the fiber product of the maps $s:\sfC_1\to \sfC_0$ and
$t:\sfC_1 \to \sfC_0$. 
\end{notation}

% The paper uses a fair amount of notation.  
In this paper there  are many Lie
2-algebras, compositions and multiplications.  For the reader's
convenience we summarize our notation below.  Some of the notation has
already been introduced above.  The explanation of the rest follows
the summary.
\subsection*{Summary of notation}
\begin{center}
\setlength{\tabcolsep}{14pt}
\renewcommand{\arraystretch}{1.5}
  \begin{longtable}{ l p{10cm}}

 $s,t:\sfC_1\to \sfC_0$ & The source and target maps of a category $\sfC$.\\
$\ast: \sfC_1 \times_{s,\sfC_0, t} \sfC_1 \to \sfC_1$ & The
composition map of a category $\sfC$.\\

$1:\sfC_0 \to \sfC_1$  & The unit map of a category $\sfC$.\\

$1_x  \in \sfC_1$ & the value of the unit map $1:\sfC_0\to \sfC_1$ on 
an object $x$ of $\sfC$.\\

$g\alpha f$ & the whiskering of a natural transformation
$\alpha:k\Rightarrow h$ by functors $g$ and $f$:
\[
\xy (-200,0)*+{g\alpha f=
\xy
(-20,0)*+{\sfD}="1";
(-5,0)*+{\sfC}="2";
( 10,0)*+{\sfB}="3";
(25,0)*+{\sfA}="4";
{\ar@{->}_{g} "2";"1"};
{\ar@{->}_{f} "4";"3"};
{\ar@/^1.pc/^{h} "3";"2"};
{\ar@/_1.pc/_{k} "3";"2"};
{\ar@{=>}^<<<{\scriptstyle \alpha} (3,3)*{};(3,-3)*{}} 
\endxy.}
\endxy
\]
\\

$G=\{G_1\toto G_0\}$ & a Lie 2-group with the Lie group $G_0$ of objects 
and $G_1$ of morphisms.     \\
$e_0\in G_0$, $e_1\in G_1$ & the multiplicative identities in the Lie 
groups $G_0$ and $G_1$ respectively.\\

$\fg=\{\fg_1\toto \fg_0\}$ & 
the Lie 2-algebra of a Lie 2-group $G$ obtained by applying the $\Lie$
functor to the objects, morphisms and the structure maps of $G$:
$\fg_0 = T_{e_0}G_0$, $\fg_1 = T_{e_1} G_1$.\\

$\scrL(G)$ & the Lie 2-algebra of a Lie 2-group $G$ whose objects are
the left-invariant vector fields on the Lie group $G_0$ and morphisms
are the left-invariant vector fields on the Lie group $G_1$.  It is
isomorphic to $\fg$.\\

$\calX (M)$ &
the Lie algebra of vector fields on a manifold $M$.\\

$\cdot$ or $m$ & the multiplication of the Lie 2-group $G$.  We may
view $m$ as a functor.  It has components $m_1:G_1\times G_1\to G_1$
and $m_0:G_0\times G_0\to G_0$.   We may abbreviate $m_1$ and $m_0$ as $m$.\\

$\bullet$ or 
$Tm:TG\times TG \to TG$ & The derivative of the multiplication functor
$m:G\times G\to G$.\\

$\star: TK_1\times_{TK_0} TK_1 \to TK$ & the composition in the tangent
groupoid $TK$ of a Lie groupoid $K$; $\star$ is the derivative of the
                                         composition\\
    & $\ast:K_1\times_{K_0}K_1\to K_1$.\\

$\X(K)$ & the Lie 2-algebra of vector fields on a Lie groupoid $K$ or the
2-vector space  underlying the Lie 2-algebra.\\

$\cL_z:Z\to Z$ & left multiplication diffeomorphism of a Lie group $Z$
defined by an element $z\in Z$: $\cL_z (z') = zz'$ for all $z'\in Z$.\\
 
$L_x: G\to G$ & the {\em functor} from a Lie 2-group $G$ to itself
defined by the left multiplication by an object $x$ of $G$.\\

$L_\gamma: L_x \Rightarrow L_y$ & the natural transformation between
two left multiplication functors defined by an arrow
$x\xrightarrow{\gamma}y$ in a Lie 2-group $G$. \\

$\lambda(x):\X(G)\to \X(G)$ & the 1-morphism of the 2-vector space
$\X(G)$ induced by an object $x$ of $G$. It is induced by the
left-multiplications functors $TL_x:TG\to TG$ and $L_{x\inv}:G\to
G$.\\

$\lambda(\gamma):\lambda(x)\Rightarrow \lambda(y)$& the 2-morphism of
the Lie 2-algebra $\X(G)$ induced by an arrow $x\xrightarrow{\gamma}y$
in the Lie 2-group $G$.\\
 \end{longtable}
\end{center}\mbox{}\\

\begin{definition}\label{def:internal}
Recall that given a category $\sfC$ with finite limits one can talk
about {\sf categories internal to $\sfC$} \cite{MacLane}.  Namely a
category $C$ internal to the category $\sfC$ consists of two objects
$C_1$, $C_0$ of $\sfC$ together with a five morphisms of $\sfC$)
$s,t: C_1 \to C_0$ (source, target), $1:C_0\to C_1$ (unit) and
composition/multiplication $\ast : C_1 \times _{C_0} C_1\to C_1$
satisfying the usual equations.  Similarly, given two categories
internal to $\sfC$ there exist {\sf internal functors} between them.
Internal functors consists of pairs of morphisms of $\sfC$ satisfying
the appropriate equations.
And given two internal functors one can talk about {\sf internal
  natural transformations} between the functors.  The categories
$\sfC$ of interest to us include groups, vector spaces, Lie groups and
Lie algebras.  The resulting internal categories are called {\sf
  2-groups} (also known as cat-groups, categorical groups,
gr-categories and categories with a group structure), Baez-Crans
{\sf 2-vector spaces}, {\sf Lie 2-groups} and {\sf Lie 2-algebras},
respectively.
\end{definition}

% \noindent
% Since we are only interested in  strict Lie 2-groups we can define
% them as follows \cite{BL}  
% Thus we use the
%following definition.
% \begin{definition}
% A {\sf Lie 2-group}  is a category internal to
% the category $\LieGp$ of Lie groups. 
% \end{definition}
We note that in particular  a Lie 2-group $G$ has a Lie group
$G_0$ of objects, a Lie group $G_1$ of morphisms and all the structure
maps: source $s:G_1\to G_0$, target $t:G_1\to G_0$, unit $1:G_0\to
G_1$ and composition $\ast: G_1\times _{s,G_0,t}G_1\to G_1$ are maps
of Lie groups (the Lie group structure on $G_2:= G_1\times
_{s,G_0,t}G_1\to G_1$ is discussed below).  We denote the
multiplicative identity of the group $G_0$ by $e_0$.  Since $1:G_0\to
G_1$ is a map of Lie groups, the multiplicative identity $e_1$ of
$G_1$ satisfies
\[
e_1 = 1_{e_0}.
\]
We denote the Lie group multiplications on $G_1$ and $G_0$ by $m_1$
and $m_0$ respectively.  Since the category of Lie groups has
transverse fiber products, the fiber product $G_2= G_1\times
_{s,G_0,t}G_1$ is a Lie group.  We denote the multiplication on
this group by $m_2$.  If we identify $G_2$ with the Lie subgroup of
$G_1\times G_1$:
\[
G_2 = \{(\sigma, \gamma)\in G_1\times G_1 \mid s(\sigma )  = t (\gamma)\},
\]
then the multiplication $m_2$ is given by the formula
\[
m_2 ((\sigma_2, \gamma_2), (\sigma_1, \gamma_1)) = (m_1(\sigma_2,
\sigma_1), m_1 (\gamma_2, \gamma_1)).
\]
Alternatively, using the infix notation $\cdot$ for the multiplications the
formula above amounts to
\[
(\sigma_2, \gamma_2)\cdot (\sigma_1, \gamma_1) = (\sigma_2\cdot \sigma_1,
\gamma_2 \cdot \gamma_1).
\]
The following lemma is well-known to experts and is easy to prove.  None the less it is crucial for many computations in the paper.

\begin{lemma}\label{lem:2.3}
  Let $G=\{G_1\toto G_0\}$ be a Lie 2-group with the composition $\ast
  :G_2 =G_1 \times _{G_0} G_1 \to G_1$ and multiplication
  $m_1:G_1\times G_1 \to G_1$, $(\gamma, \sigma) \mapsto
  \gamma\cdot \sigma$.  Then
\begin{equation}\label{eq:1}
  (\sigma_2\ast \sigma_1)\cdot (\gamma_2\ast \gamma_1) 
= (\sigma_2\cdot \gamma_2)\ast (\sigma_1\cdot \gamma_1),
\end{equation}
for all $(\sigma_2, \sigma_1),(\gamma_2, \gamma_1) \in G_2 = G_1\times
_{s,G_0,t}G_1$.
\end{lemma}
\begin{proof}
Since the composition $\ast :G_2 \to G_1$ is a Lie
group homomorphism,
\begin{equation}\label{eq:2}
 \ast((\sigma_2, \sigma_1)\cdot (\gamma_2,\gamma_1)) = 
(\ast(\sigma_2, \sigma_1))\cdot(\ast(\gamma_2,\gamma_1)).
\end{equation}
On the other hand 
\begin{equation}\label{eq:3}
(\sigma_2, \sigma_1)\cdot(\gamma_2,\gamma_1) 
= (\sigma_2\cdot\gamma_2, \sigma_1\cdot\gamma_1)
\end{equation}
while 
\begin{equation}\label{eq:4}
  (\ast(\sigma_2, \sigma_1))\cdot(\ast(\gamma_2,\gamma_1)) 
\equiv  (\sigma_2\ast \sigma_1)\cdot(\gamma_2\ast \gamma_1) 
\end{equation}
when we switch from the prefix to infix notation.  Similarly,
\begin{equation}\label{eq:5}
 \ast(\sigma_2\cdot\gamma_2, \sigma_1\cdot\gamma_1)
\equiv (\sigma_2\cdot\gamma_2)\ast( \sigma_1\cdot\gamma_1).
\end{equation}
Therefore
\[
(\sigma_2\cdot\gamma_2)\ast( \sigma_1\cdot\gamma_1) 
= (\sigma_2\ast \sigma_1)\cdot(\gamma_2\ast \gamma_1). 
\] 
\end{proof}

\begin{corollary}
  The multiplications $m_i:G_i\times G_i \to G_i$, $i=0,1$ on a Lie
  2-group $G$ assemble into a functor $m:G\times G \to G$.
\end{corollary}
\begin{proof}
Omitted.
\end{proof}
Equation \eqref{eq:1} also implies that the multiplication functor
$m:G\times G \to G$ and the composition homomorphism $\ast :G_1\times
_{G_0} G_1 \to G_1$ in a Lie 2-group $G$ are closely related.  In fact
they determine each other \cite{MacLane}.  For the convenience of the
reader we recall a proof that the multiplication functor $m$
determines the composition homomorphism $\ast$:

\begin{lemma}\label{lem:1}
  For any two composable arrows $\sigma, \gamma$ of a Lie 2-group
  $G$ with $s(\sigma)= b = t(\gamma)$
\[
\sigma \ast\gamma = \gamma\cdot 1_{b\inv} \cdot \sigma.
\]
Here as before $s,t:G_1\to G_0$ are the source and target maps, $1_{b\inv}$
denotes the unit arrow at the object $b\inv$ of $G$, $\cdot$ stands for the
multiplication $m_1$ on the space of arrows $G_1$ of the Lie 2-group
$G$ and $\ast: G_1\times_{G_0} G_1\to G_1$ is the composition homomorphism.
\end{lemma}

\begin{proof}
  We follow the proof in \cite[p.\ 186]{MacLane}. 
Note that since $1:G_0\to
G_1$ is a homomorphism,  the inverse $1_b \inv$ of $1_b$ with  respect to the multiplication $m_1$ is $1_{b\inv}$.  We compute
\begin{eqnarray*}
  \sigma \ast \gamma &=&
  ((1_b \cdot (1_b\inv\cdot \sigma)) \ast (\gamma \cdot (1_b \inv \cdot 1_b))\\
  &=& (1_b \ast \gamma) \cdot ((1_b \inv \cdot \sigma)\ast (1_b \inv \cdot 1_b))
  \quad \textrm{ by \eqref{eq:1} }\\
  &=& \gamma \cdot (1_b \inv \ast 1_b \inv )\cdot (\sigma \ast 1_b )
\qquad \textrm{ by \eqref{eq:1} again}\\
  &=& \gamma \cdot 1_{b \inv} \cdot \sigma\quad \textrm{since } 
  1_x \ast 1_x = 1_x 
  \textrm{ for all }x\in G_0\textrm{ and } (1_b)\inv = 1_{b\inv} .
\end{eqnarray*}
\end{proof}
Lemma~\ref{lem:1} has a well-known corollary:
any Lie 2-group is a Lie groupoid.
In fact we can be more precise:
\begin{corollary} \label{cor:2.7}
  A Lie 2-group $G$ is isomorphic, as a category internal to the
  category of manifolds, to the action groupoid $\{K\times G_0\toto
  G_0\}$ where $K$ is the kernel of the source map $s:G_1\to G_0$ and
  the action of $K$ on $G_0$ is given by
\[
k\diamond x := t(k)\cdot x
\]
for all $(k,x) \in K\times G_0$. As before $t:G_1\to G_0$ is the
target map.
\end{corollary} 
\begin{proof}[Sketch of proof]
  The isomorphism of categories $\varphi:G\to \{K\times G_0\to G_0\}$ 
  is defined to be identity on objects.  On arrows $\varphi$  is given by
\[
\varphi_1(y\xleftarrow{\gamma} x) = (\gamma \cdot 1_{x\inv}, x).
\]
\end{proof}

\begin{remark}
The same argument shows that any 2-group (i.e., a category internal to
the category of groups) is an action groupoid.
 \end{remark} 
We next recall the definitions  of the 2-categories of Lie 2-algebras
and  of 2-vector spaces.
\begin{definition}
  Lie 2-algebras naturally form a strict 2-category $\LietAlg$.  The
  objects of this 2-category are Lie 2-algebras, the 1-morphisms are
  functors internal to the category $\mathsf{LieAlg}$ of Lie algebras and 2-morphisms
    are internal natural transformations. % Alternatively one can
%   localize the 2-category $\LietAlg_{{strict}}$ at essential equivalences and
%   obtain a bicategory $\LietAlg$.  See \cite{B-E_L} for more details.
 \end{definition}
\begin{definition}
  2-vector spaces naturally form a strict 2-category $2\Vect$.  The
  objects of this 2-category are 2-vector spaces.  The 1-morphisms of
  $2\Vect$ are internal functors and 2-morphisms are internal natural
  transformations.
\end{definition}

\begin{remark}
  There is an evident forgetful functor $U:\LietAlg\to
  2\Vect$.  We will suppress this functor in our notation and will use
  the same symbol for a Lie 2-algebra and its image under the functor
  $U$, that is, its underlying 2-vector space.
  \end{remark}\mbox{}

\subsection{The Lie 2-algebra $\X(K)$ of multiplicative vector fields
  on a Lie groupoid $K$} \mbox{}\\[10pt]
In this subsection we recall some of the results of \cite{B-E_L}.  We
start by recalling the definition of the {\em category} of
multiplicative vector fields $\X(K)$ on a Lie groupoid $K$, which is
due to Hepworth \cite{Hepworth}. 

\begin{definition}
 A {\sf multiplicative vector field}
on a Lie groupoid $K=\{K_1\toto K_0\}$ is a functor $v:K\to TK$ so
that $\pi_K\circ v = \id_K$.  A {\sf morphism} (or an {\sf arrow})
from a multiplicative vector field $v$ to a multiplicative vector
field $w$ is a natural transformation $\alpha:v\Rightarrow w$ so that
$\pi_K (\alpha (x)) = 1_x$ for any object $x$ of the groupoid $K$.
\end{definition}

Multiplicative vector fields and morphisms between them are easily
seen to form a category: the composite of two morphisms
$\alpha:v\Rightarrow w$ and $\beta:w\Rightarrow u$ is the natural
transformation $\beta \circ_v \alpha$, where $\circ_v$ denotes the
vertical composition of natural transformations.  That is, for any
object $x\in K_0$
\[
(\beta \circ_v \alpha) (x) = \beta(x) \star \alpha (x)
\]
where as before $\star: TK_1 \times _{TK_0} TK_1 \to TK_1$ is the
derivative of the composition $\ast: K_1 \times _{K_0} K_1 \to K_1$.
Since $\pi_K:TK \to K$ is functor,
\[
\pi_K (\beta(x) \star \alpha (x)) = 
\pi_K(\beta(x))\ast \pi_K (\alpha(x)) = 1_x \ast 1_x = 1_x 
\]
for all $x\in K_0$.  Hence $\beta \circ_v \alpha$ is a morphism from
$v$ to $u$. 

It is not hard to see that the collection $\X(K)_0$ multiplicative
vector fields form a vector space \cite{MackXu}.  It is a little harder
to see that $\X(K)_0$ is a Lie algebra ({\em op.\ cit.}).  However,
the Lie bracket on $\X(K)_0$ is easy to describe.  A multiplicative
vector field $u:K\to TK$ is, in particular,  a pair of ordinary vector fields:
\[
u= (u_0:K_0\to TK_0, u_1:K_1 \to TK_1).
\]
The bracket on $\X(K)_0$ is defined by 
\[
[(u_0, u_1), (v_0, v_1)]:= ([u_0,v_0], [u_1, v_1]).
\]
To see that the definition makes sense one checks that $([u_0,v_0],
[u_1, v_1])$ is a functor from $K$ to $TK$; see \cite{MackXu}.

The space of arrows $\X(K)_1$ is a Lie algebra as well and the
structure maps of the category $\X(K)$ are Lie algebra maps
\cite{B-E_L}.  In other words the category $\X(K)$ underlies a Lie
2-algebra.  

The bracket on the elements of $\X(K)_1$ ultimately comes
from the Lie bracket on the vector fields on the manifold $K_1$
\cite{B-E_L}.    But the relationship is not direct since the elements
of $\X(K)_1$ are not vector fields.  In more detail, write an arrow $\alpha\in \X(K)_1$ as 
\[
\alpha = (\alpha -\mathbbm{1}_{\mathbbm{s}(\alpha)})+  \mathbbm{1}_{\mathbbm{s}(\alpha)},
\]
where  $\mathbbm{1}:\X(K)_0 \to \X(K)_1$ is the unit map and
$\mathbbm{s}:\X(K)_1\to \X(K)_0$ is the source map of the category
$\X(K)$.  Recall that for a multiplicative vector field
$X$, the morphism $\bbm{1}_X:X\Rightarrow X$ is defined by
\[
\bbm{1}_X (x) = T1(X_0(x))
\]
for all $x\in K_0$.  The multiplicative vector field $\bbm{s}(\alpha)$
satisfies
\[
(\bbm{s}(\alpha))_0(x) = Ts (\alpha(x))
\]
for all $x\in K_0$, where on the right hand side $s:K_1\to K_0$ is, as
before, the source map for the Lie groupoid $K$.  Then
\[
Ts (\alpha -\mathbbm{1}_{\mathbbm{s}(\alpha)}) =0,
\]
hence $\alpha -\bbm{1} _{\bbm{s}(\alpha)}$ is a section of the Lie
algebroid $A_K\to K_0$ of the Lie groupoid $K$. 

Recall that the Lie bracket on the space of sections $\Gamma(A_K)$ of
the Lie algebroid $A_K$ is constructed by embedding %embedding
$\Gamma(A_K)$ into the space of vector fields on $K_1$ as
right-invariant vector fields.  That is, one constructs a map
\[
j:\Gamma(A_K) \to \calX(K_1)
\]
by setting 
\begin{equation} \label{eq:j}
j(\sigma) \, (\gamma) := TR_\gamma \left(\sigma (t(\gamma))\right)
\end{equation}
for all $\gamma \in K_1$.  The map $R_\gamma: s\inv (t(\gamma)) \to
K_1$ is defined by composition with $\gamma$ on the right:
\[
R_\gamma (\mu) := \mu \ast \gamma
\]
for all $\mu \in K_1$ with $s(\mu) = t(\gamma)$.

We now recall the construction of a Lie algebra structure on the space
$\X(K)_1$ (see \cite{B-E_L} where the details of the construction are
phrased somewhat differently).   Define
\[
J: \X(K)_1 \to \calX(K_1)
\]
by setting
 \begin{equation}\label{eq:J}
J(\alpha) := j(\alpha  -\bbm{1}_{\bbm{s}(\alpha)}) + \bbm{s}(\alpha)_1.
\end{equation}
The map $J$ is injective and its image happens to be closed under the
Lie bracket.  So for $\alpha, \beta\in \X(K)_1$  
we can (and do) define the Lie bracket
$[\alpha, \beta]$ to be the unique element of $\X(K)_1$ with
\[
J([\alpha, \beta]) = [J(\alpha), J(\beta)].
\]
One checks that the category $\X(K)$ of multiplicative vector fields
with the Lie algebra structures on the spaces of objects and morphisms
does form a Lie 2-algebra; see \cite{B-E_L}.

\section{Actions  and 
  representations of Lie 2-groups}\label{sec:action}

The goal of this section is to construct a representation
$\lambda:G\to \mathrm{GL} (\X(G))$ of a Lie group $G$ on its 2-vector
space $\X(G)$ of vector fields induced by the action of $G$ on itself
by left multiplication. This is the representation briefly described
in the introduction. We start by recalling some well-known material
about actions of Lie 2-groups. 
Recall that a {\sf 2-group} is a category internal to the category of
groups and a {\sf homomorphism}  of 2-groups is a functor internal to the category of groups (cf.\ Definition~\ref{def:internal}).

\begin{definition}[the 2-group $\Aut(K)$] Let $K$ be a Lie groupoid.
  The 2-group $\Aut(K)$ of automorphisms of $K$ is defined as follows.  

The group of objects $\Aut(K)_0$ consists of strictly
invertible smooth (i.e., internal) functors $f:K\to K$.  The group operation on 
 $\Aut(K)_0$ is the composition of functors.
The group of morphisms $\Aut(K)_1$
is the group of (smooth) natural isomorphisms under {\em vertical} composition.
The composition homomorphism $*: \Aut(K)_1\times_{\Aut(K)_0}
\Aut(K)_1\to \Aut(K)_1$ is the {\em horizontal} composition of
natural isomorphisms.  There are also evident source, target and unit maps:
\[
s(f\stackrel{\alpha}{\Rightarrow}g) = f, \quad
t(f\stackrel{\alpha}{\Rightarrow}g) =g, \quad 1 (f) =
(f\stackrel{\id_f}{\Rightarrow}f).
\]
Note that the component of $\id_f$ at an object $x\in K_0$ is
\[
 \id_f (x) = 1_{f(x)},
\]
the unit arrow on the object $f(x)$ of $K$.
\end{definition}

\begin{definition}
  A (strict left) {\sf action} of a Lie 2-group $G$ on a Lie groupoid $K$ is a
  functor $\ba:G\times K\to K$ so that the two diagrams
\begin{equation} \label{eq:3.2.1}
\xy
(-15,10)*+{G\times G\times K}="1";
(15,10)*+{G\times K}="2";
( -15,-10)*+{G\times K}="3";
(15,-10)*+{K}="4";
{\ar@{->}^{\id_G \times \ba} "1";"2"};
{\ar@{->}_{m\times \id_K} "1";"3"};
{\ar@{->}^{\ba} "2";"4"};
{\ar@{->}_{\ba} "3";"4"};
\endxy
\qquad \textrm{ \quad and \quad}\quad 
\xy
(-15,10)*+{G\times K}="1";
(15,10)*+{K}="2";
( -15,-10)*+{K}="3";
%(15,-10)*+{K}="4";
{\ar@{->}^{\ba} "1";"2"};
{\ar@{->}^{e\times \id_K} "3";"1"};
{\ar@{=} "3";"2"};
\endxy
\end{equation}
commute.  Here as before $m:G\times G\to G$ is the multiplication
functor. The functor $e\times \id_K$ is defined by $(e\times
\id_K)\,(\sigma) = (e_1, \sigma)$ for all arrows $\sigma $ of $K$, where
as before $e_1\in G_1$ is the multiplicative identity.
\end{definition}

\begin{notation}
  Given an action $\ba:G\times K\to K$ a Lie 2-group $G$ on a Lie
  groupoid $K$ it will be convenient at times to abbreviate $\ba(x,b)$
  as $x\cdot b$ for any two objects $x$ of $G$ and $b$ of $K$.
  Similarly we abbreviate $\ba(\gamma, \sigma)$ as $\gamma\cdot \sigma$
  for arrows $\gamma$ of $G$ and $\sigma$ of $K$. 
\end{notation}

\begin{remark}
  In the notation above the fact that $\ba:G\times K\to K$ preserves the
  composition of arrows
  translates into
\begin{equation}
(\gamma_2\ast\gamma_1)\cdot (\sigma_2 \ast \sigma_1) =
(\gamma_2\cdot \sigma_2)\ast (\gamma_1 \cdot \sigma_1)
\end{equation}
for any two pairs of composable arrows $(\gamma_2,\gamma_1)\in G_1
\times_{G_0} G_1$ and $(\sigma_2, \sigma_1)\in K_1\times_{K_0}K_1$.
\end{remark}

\begin{lemma} \label{lem:3.2.1}
An action $\ba:G\times K \to K$ of a Lie 2-group $G$ on a
  Lie groupoid $K$ gives rise to  a homomorphism of 2-groups
\begin{equation}\label{eq:3.1.1}
\ha:G\to \Aut(K).
\end{equation}
In particular for each object $x\in G_0$ there is a functor
$\ha(x):K\to K$ satisfying
\[
\ha(x)\, (b\xleftarrow{\sigma}a):= x\cdot b\xleftarrow{1_x\cdot \sigma}
x\cdot a
\]
for all arrows $b\xleftarrow{\sigma}a$ of the groupoid $K$.  And for
each arrow $y\xleftarrow{\gamma}y$ of $G$ there is a natural
transformation $\ha(\gamma):\ha(x)\Rightarrow \ha(y)$ satisfying
\[
\ha(\gamma) (b) = \gamma \cdot 1_b
\]
for all objects $b$ of $K$.
\end{lemma}

\begin{remark} The functor
\eqref{eq:3.1.1} is a homomorphism of 2-groups if and only if
\begin{enumerate}
\item $\ha(e_0\xleftarrow{e_1}e_0) =\,\,
  (\id_K\stackrel{1_{\id_K}}{\Leftarrow} \id_K)$ and
\item $\ha(\gamma_2 \cdot \gamma_1) = \ha(\gamma_2)\circ
  _{hor} \ha(\gamma_1)$ for any pairs of arrows $\gamma_2,
  \gamma_1$ of $G$. (Here as before $\cdot$ denotes the multiplication in the
  Lie group $G_1$.)
\end{enumerate} 
\end{remark}

\begin{remark}\label{rmrk:3.7.1}
Recall that given four functors and two natural
  transformations as below
\[
\xy
(-45,0)*+{\mathsf{C}}="1";
(-15, -0)*+{\mathsf{B}}="2";
( 15, 0)*+{\mathsf{A}}="3";
{\ar@/^1.pc/^{k} "2";"1"};
{\ar@/_1.pc/_{n} "2";"1"};
{\ar@/^1.pc/^{g} "3";"2"};
{\ar@/_1.pc/_{h} "3";"2"};
{\ar@{=>}^<<<{\scriptstyle \alpha} (0, 2.5)*{};(0, -2.5)*{}};
{\ar@{=>}^<<<{\scriptstyle \beta} (-30, 2.5)*{};(-30, -2.5)*{}};
\endxy
\]
the component $(\beta\circ_{hor}\alpha)(a)$ of the horizontal
composition  of $\beta$ and $\alpha$ at an object $a\in \sfA_0$ is given by
\[
 (\beta\circ_{hor}\alpha)(a) = 
\beta_{g(a)} \ast n (\alpha_a)
\]
where  $\ast: \sfC_1 \times _{\sfC_0}\sfC_1\to \sfC_1$ is the
composition in the category $\sfC$.
\end{remark}

\begin{proof}[Proof of Lemma~\ref{lem:3.2.1}]
  Since $\ba:G\times K\to K$ is a functor, for any two composable arrows
  $\sigma_2, \sigma_1$ in $K$ and for any object $x$ of $G$
\[
\ba(1_x, \sigma_2\ast \sigma_1) = \ba ((1_x, \sigma_2)\ast (1_x, \sigma_1)) 
= \ba(1_x, \sigma_2)\ast \ba(1_x, \sigma_1).
\]
We also have $\ba(1_x, \sigma_2\ast \sigma_1) = \ba ((1_x, \sigma_2)\ast
(1_x, \sigma_1))$ and $\ba(1_x, \sigma_2)\ast \ba(1_x, \sigma_1) =
\ha(x)(\sigma_2)\ast \ha(x)(\sigma_1)$.  Hence
\[
\ha(x)(\sigma_2\ast \sigma_1) = 
\ha(x)(\sigma_2)\ast\ha(x)(\sigma_1).
\]
We conclude that $\ha(x)$ is a functor for all objects $x$ of the
2-group $G$.

To check that for an arrow $x\xrightarrow{\gamma} y$ in $G$,
$\ha(\gamma)$ is a natural transformation from the functor
$\ha(x)$ to the functor $\ha(y)$ we need to check that for any
arrow $b\xleftarrow{\sigma} a$ in $K$
\begin{equation} \label{eq:3.3.3}
\ha(\gamma) (b)\ast \ha(x)(\sigma) = \ha(y)
(\sigma)\ast\ha(\gamma) (b).
\end{equation}
Now 
\begin{eqnarray*}
  \ha(\gamma) (b)\ast \ha(x)(\sigma) &=& (\gamma\cdot 1_b )\ast (1_x \cdot \sigma)\\
  &=& (\gamma \ast 1_x) \cdot (1_b\ast \sigma)\qquad 
\textrm{ (since } a \textrm{ is a functor) }\\ 
  &=& \gamma \cdot \sigma.
\end{eqnarray*}
Similarly
\[
 \ha(y)
(\sigma)\ast\ha(\gamma) (b)= \gamma \cdot \sigma
\]
as well.  Hence \eqref{eq:3.3.3} holds and  $\ha(\gamma)$ {\em is} a natural transformation.
Since $K$ is a groupoid $\ha(\gamma)$ is a natural isomorphism.

It is easy to see that $\ha(e_0)$ is the identity functor $\id_K$
and that $\ha(e_1)$ is the identity natural isomorphism
$1_{\id_K}$.

To prove that $\ha$ is a homomorphism of 2-groups it remains to check that 
\begin{equation}
\ha (\gamma_2\cdot \gamma_1) = \ha (\gamma_2)\circ_{hor}\ha (\gamma_1)
\end{equation}
for all arrows $\gamma_2,\gamma_1$ of $G$.
This is a computation.  Fix an object $a$ of $K$.  Then
\begin{eqnarray*}
\ha (\gamma_2)\circ_{hor}\ha (\gamma_1)&=& \ha(\gamma_2)(\ha(\gamma_1) a)
\ast (\ha(x_2)\left( \ha(\gamma_1)a\right)\qquad 
\textrm{ (by Remark~\ref{rmrk:3.7.1}) }  \\
&=&(\gamma_2 \cdot 1_{y_2\cdot a})\ast \left(1_{x_2}\cdot (\gamma_1\cdot 1_a)\right)
\qquad \textrm{(by definition of }\ha)\\
&=& \gamma_2\cdot (\gamma_1 \cdot 1_a) = (\gamma_2\cdot ]\gamma_1)\cdot 1_a
\qquad \textrm{ since the left diagram in \eqref{eq:3.2.1} commutes})\\
&=& \ha(\gamma_2\cdot\gamma_1)(a).
\end{eqnarray*}
\end{proof}

\begin{corollary} \label{lem:2.2.0}
For any Lie 2-group $G$  there is a homomorphism of 2-groups
\begin{equation}
L:G\to \Aut(G),\qquad 
(x\xrightarrow{\gamma}y)\mapsto (L_x\stackrel{L_\gamma}{\Rightarrow}L_y)
\end{equation}
where the smooth functors $L_x:G\to G$ are defined by
\[
L_x (\sigma) = 1_x \cdot \sigma 
\]
and the natural isomorphisms $L_\gamma:L_x\Rightarrow L_y$ are defined by
\[
L_\gamma(a) = \gamma \cdot 1_a
\]
for all objects $a$ of $G$.
Here $\cdot$ denotes the multiplication in the group $G_0$ and in the
group $G_1$.
\end{corollary}

\begin{proof}
  The multiplication functor $m:G\times G\to G$ is an action of the
  Lie 2-group $G$ on the Lie groupoid $G$.  Now apply Lemma~\ref{lem:3.2.1}.
\end{proof}

\begin{lemma}
  Let $G$ be a 2-group and $K$ a Lie groupoid.  A homomorphism
  $\rho:G\to \Aut(K)$ induces a homomorphism
\[
T\rho:G\to \Aut(TK), \qquad T\rho (x\xrightarrow{\gamma}y) = 
T\rho(x)\stackrel{T\rho(\gamma)}{\Rightarrow}T\rho(y).
\]
\end{lemma}
\begin{proof}
  The homomorphism $T\rho$ is obtained by composing the functor $\rho$
  with the tangent 2-functor $T:\LieGpd\to \LieGpd$.
\end{proof}
\begin{notation}
We denote the 2-vector space underlying the Lie 2-algebra of vector
fields on a Lie groupoid $K$ by the same symbol $\X(K)$.  
\end{notation}
\begin{definition}[the 2-group $\mathrm{GL}(V)$]
  Let $V$ be a 2-vector space.  We define the 2-group $\mathrm{GL}(V)$ of
  automorphisms of a 2-vector space $V$ as follows.
The group of objects $\mathrm{GL}(V)_0$ consists of strictly
invertible linear  (i.e., internal) functors $f:V\to V$.  The group operation on 
 $\mathrm{GL}(V)_0$ is the composition of functors.
The group of morphisms $\mathrm{GL}(V)_1$
is the group of internal natural isomorphisms under  vertical composition.
The composition homomorphism $*: \mathrm{GL}_1\times_{ \mathrm{GL}_0}
 \mathrm{GL}_1\to  \mathrm{GL}_1$ is the  horizontal  composition of
natural isomorphisms.  There are also evident source, target and unit maps:
\[
s(f\stackrel{\alpha}{\Rightarrow}g) = f, \quad
t(f\stackrel{\alpha}{\Rightarrow}g) =g, \quad 1 (f) =
(f\stackrel{\id_f}{\Rightarrow}f).
\]
\end{definition}

\begin{lemma} \label{lem:2.2.1} Let $G$ be a Lie 2-group and $K$ a Lie
  groupoid.  A homomorphism $\varphi:G\to \Aut(K)$ of 2-groups (i.e.,
  a functor internal to the category of groups) gives
  rise to a homomorphism of 2-groups
\[
\Phi:G\to \mathrm{GL}(\X(K)),
\]
a representation of the 2-group $G$ on the 2-vector space of vector
fields on the Lie groupoid $K$.
\end{lemma}

\begin{proof}
  As a first step given an object $x$ of $G$ we would like to define a
  functor $\Phi(x):\X(K)\to \X(K)$ by setting
\[
\Phi (x) (v\stackrel{\alpha}{\Rightarrow} w) = \xy
(-20,0)*+{TK}="1";
(-5,0)*+{TK}="2";
( 10,0)*+{K}="3";
(25,0)*+{G}="4";
{\ar@{->}_{T\varphi(x)} "2";"1"};
{\ar@{->}_{\varphi({x\inv})} "4";"3"};
{\ar@/^1.pc/^{w} "3";"2"};
{\ar@/_1.pc/_{v} "3";"2"};
{\ar@{=>}^<<<{\scriptstyle \alpha} (3,3)*{};(3,-3)*{}} 
\endxy.
\]
for all arrows $v\stackrel{\alpha}{\Rightarrow}w$ in the 2-vector space $\X(K)$.
An object of $\X(K)$ is a functor $v:K\to TK$ with $\pi\circ v =
\id_K$.  Since $\varphi(x\inv )$ and $T\varphi(x)$ are functors,
\[
\Phi(x)v:= T\varphi(x)\circ v \circ \varphi(x\inv)
\]
is a functor.  Moreover 
\begin{eqnarray*}
\pi \circ (\Phi(x)v)&=& \pi \circ  T\varphi(x)\circ v \circ \varphi(x\inv)\\
&=& \varphi(x) \circ \pi \circ v \circ \varphi(x\inv)\qquad 
\textrm{ since }\pi \circ T\varphi = \varphi \circ \pi \\
&=&\varphi(x)\circ \id_K \circ \varphi(x\inv) = \id_K.
\end{eqnarray*}
Hence $\Phi(x)v$ is an object of $\X(K)$ for all $x\in G_0$ and all
$v\in \X(K)_0$.

An arrow in $\X(K)$ from an object $v$ to an object $w$ is a natural
transformation $\alpha:v\Rightarrow w$ with $\pi \alpha =1_{\id_K}$.
Now since $\Phi(x)\alpha$ is obtained from a natural transformation
$\alpha$ by whiskering with functors (namely $\Phi(x)\alpha
=T\varphi(x)\alpha \varphi(x\inv)$), $\Phi(x)\alpha$ is a natural
transformation from $\Phi(x)v$ to $\Phi(x)w$.  Additionally
\begin{eqnarray*}
\pi(\Phi(x)\alpha)&=& \pi T\varphi(x) \alpha \varphi(x\inv)\\
&=& \varphi(x) \pi  \alpha \varphi(x\inv)\\
&=& \varphi(x)1_{\id_K}\varphi(x\inv) = 1_{\id_K}.
\end{eqnarray*}
Hence $\Phi(x)\alpha$ is an arrow in the 2-vector space $\X(K)$.
Finally the purported functor $\Phi(x)$ preserves composition of
arrows because whiskering by functors commutes with the vertical
composition of natural transformations.  
We conclude that $\Phi(x):\X(K)\to \X(K)$ is a well-defined functor.   

Since the components $T\varphi(x)_0:TK_0 \to TK_0$ and
$T\varphi_1:TK_1\to TK_1$ are fiberwise linear, for any scalars
$c,d\in \R$ and any two multiplicative vector fields $v,w:K\to TK$
\[
\Phi(x)( c v + dw) = c \Phi(x) v + d \Phi(x) w.
\]
Similarly for any two arrows $\alpha_1:v_1\Rightarrow w_1$,
$\alpha_2:v_2\Rightarrow w_2$, any two scalars $c_1, c_2\in \R$ and any
object $a$ of $K$
\begin{eqnarray*}
\Phi(x)(c_1 \alpha_1+ c_2 \alpha_2)\, (a)
&=& T\varphi(x) (c_1 \alpha_1 (\varphi(x\inv )(a) 
+c_2 \alpha_2 (\varphi(x\inv )(a))\\ 
&=& c_1 (T\varphi(x) \alpha_1 \varphi(x\inv ))(a) 
+c_2 ( T\varphi(x)\alpha_2 \varphi(x\inv ))(a)\\
&=& (c_1 \Phi(x)\alpha_1 + c_2 \Phi(x)\alpha_2)(a).
\end{eqnarray*}
We conclude that $\Phi(x)$ is a 1-morphism of 2-vector spaces.\\

Given an arrow $x\xleftarrow{\gamma} y$ in $G$ we would like to define
a natural transformation $\Phi(\gamma):\Phi(x)\Rightarrow \Phi(y)$ by
setting
\[
\Phi(\gamma)v:= 
\xy
(-30,0)*+{TK}="1";
(-5,0)*+{\,\,TK}="2";
( 10,0)*+{K}="3";
(35,0)*+{K}="4";
{\ar@{->}_{v} "3";"2"};
{\ar@/^1.pc/^{T\varphi(x)} "2";"1"};
 {\ar@/_1.pc/_{T\varphi(y)} "2";"1"};
 {\ar@/^1.pc/^{\varphi(x\inv)} "4";"3"};
{\ar@/_1.pc/_{\varphi(y\inv)} "4";"3"};
{\ar@{=>}^<<<<{\scriptstyle T\varphi(\gamma)} (-20,3)*{};(-20,-3)*{}}; 
{\ar@{=>}^<<<<{\scriptstyle \varphi(\gamma\inv)} (19,3)*{};(19,-3)*{}}; 
\endxy
\]
for all multiplicative vector fields $v:K\to TK$.  By construction
$\Phi(\gamma)v$ is a natural transformation from $T\varphi(x)\circ v
\circ \varphi(x\inv) = \Phi(x)v$ to $T\varphi(y)\circ v \circ
\varphi(y\inv) = \Phi(y)v$.  Moreover
\begin{eqnarray*}
\pi (\Phi(\gamma)v)&=& \pi T\varphi(\gamma)v\varphi(\gamma\inv)\\
&=& \varphi(\gamma) (\pi \circ v) \varphi(\gamma\inv)\\
&=& \varphi(\gamma)\circ_{vert} 1_{\id_K}\circ_{vert} \varphi(\gamma\inv)\\
&= & 1_{\id_K}  \qquad (\textrm{ since }\varphi \textrm{ is a homomorphism}).
\end{eqnarray*}
We conclude that for any multiplicative vector field $v$ the natural
transformation $\Phi(\gamma)v$ is an arrow in the 2-vector space
$\X(K)$.   

It is easy to check that $\Phi(\gamma):\X(K)_0 \to \X(K)_1$ is linear.
We now check that $\Phi(\gamma)$ is an actual natural transformation
from $\Phi(x)$ to $\Phi(y)$. That is, we check that for any arrow
$v\stackrel{\alpha}{\Rightarrow}w$ in $\X(K)$ the diagram
\[
\xy
(-15,10)*+{\Phi(x)w}="1";
(15,10)*+{\Phi(x)v}="2";
( -15,-10)*+{\Phi(y)w}="3";
(15,-10)*+{\Phi(y)v}="4";
{\ar@{=>}^{\Phi(x)\alpha} "2";"1"};
{\ar@{=>}_{\Phi(\gamma)w} "1";"3"};
{\ar@{=>}^{\Phi(\gamma)v} "2";"4"};
{\ar@{=>}_{\Phi(y)\alpha} "4";"3"};
\endxy
\]
commutes in the category $\X(K)$.  By definition the composition of the arrows
$\Phi(\gamma)w$ and $ \Phi(x)\alpha$ is the vertical composition
of the diagrams
\[
\xy
(-20,0)*+{TK}="1";
(-5,0)*+{TK}="2";
( 10,0)*+{K}="3";
(25,0)*+{G}="4";
{\ar@{->}_{T\varphi(x)} "2";"1"};
{\ar@{->}_{\varphi({x\inv})} "4";"3"};
{\ar@/^1.pc/^{w} "3";"2"};
{\ar@/_1.pc/_{v} "3";"2"};
{\ar@{=>}^<<<{\scriptstyle \alpha} (3,3)*{};(3,-3)*{}} 
\endxy
\]
and 
\[
\xy
(-30,0)*+{TK}="1";
(-5,0)*+{\,\,TK}="2";
( 10,0)*+{K}="3";
(35,0)*+{K}="4";
{\ar@{->}_{w} "3";"2"};
{\ar@/^1.pc/^{T\varphi(x)} "2";"1"};
 {\ar@/_1.pc/_{T\varphi(y)} "2";"1"};
 {\ar@/^1.pc/^{\varphi(x\inv)} "4";"3"};
{\ar@/_1.pc/_{\varphi(y\inv)} "4";"3"};
{\ar@{=>}^<<<<{\scriptstyle T\varphi(\gamma)} (-20,3)*{};(-20,-3)*{}}; 
{\ar@{=>}^<<<<{\scriptstyle \varphi(\gamma\inv)} (19,3)*{};(19,-3)*{}}; 
\endxy
\]
which is 
\[
T\varphi(\gamma)\circ_{hor} \alpha \circ _{hor} \varphi(\gamma\inv).
\]
Similarly
\[
\Phi(y)\alpha\circ_{vert} \Phi(\gamma)\alpha 
= T\varphi(\gamma)\circ_{hor} \alpha \circ _{hor} \varphi(\gamma\inv)
\]
as well.  Therefore $\Phi(\gamma):\Phi(x)\Rightarrow \Phi(y)$ is a
2-morphism of 2-vector spaces.\\

We finish the proof by checking that $\Phi$ is a homomorphism of
2-groups.  Clearly $\Phi(e_0) = \id_{\X(K)}$ and $\Phi(1_{e_0}) = 1_{
  \id_{\X(K)}}$.  For any two objects $x_2, x_1$ of $G$
$\varphi(x_2\cdot x_1) = \varphi(x_2)\circ \varphi(x_1)$ and $T\varphi(x_2\cdot x_1) = T\varphi(x_2)\circ T\varphi(x_1)$ 
Consequently for any multiplicative vector field $v$
\begin{eqnarray*}
  \Phi(x_2\cdot x_1)v 
&=& T(\varphi(x_2\cdot x_1)\circ v \circ \varphi((x_2\cdot x_1)\inv)\\
&=&T\varphi(x_2)\circ T\varphi(x_1)\circ v \circ 
\varphi(x_2\inv)\circ \varphi(x_1\inv)\\
&=& \Phi(x_2)(\Phi(x_1)v).
\end{eqnarray*}
Checking that $\Phi(\gamma_2\cdot \gamma_1) = \Phi(\gamma_2)\circ
_{hor} \Phi(\gamma_1)$ is a bit more involved.
Note first that 
\[
\varphi (\gamma_2 \cdot \gamma_1) = \varphi(\gamma_2) \circ_{hor}
\varphi(\gamma_1)
\]
since $\varphi$ is a homomorphism (1-morphism) of 2-groups.   Similarly
\[
T\varphi (\gamma_2 \cdot \gamma_1) =
 T\varphi(\gamma_2) \circ_{hor} T\varphi(\gamma_1) .
\]
Recall that the arrows in the category $\mathrm{GL}(\X(K))$ are
natural isomorphisms, and that the composition of arrows in
$\mathrm{GL}(\X(K))$ is the vertical composition.  Hence by
Remark~\ref{rmrk:3.7.1} for any object $u$ of $\X(K)$,
\[
(\Phi(\gamma_2)\circ_{hor}\Phi(\gamma_1))(u)
= \left(\Phi(\gamma_2)(\Phi(y_1) u)   \right) \circ_{vert}
 \left(\Phi(x_2)(\Phi(\gamma_1) u)   \right).    
\]
Since
\[
(\Phi(x_2)(\Phi(\gamma_1) u) \quad =\qquad
\xy
(-50,0)*+{TK}="0";
(-30,0)*+{TK}="1";
(-5,0)*+{\,\,TK}="2";
( 10,0)*+{K}="3";
(35,0)*+{K}="4";
(54,0)*+{K}="5";
{\ar@{->}_{T\varphi (x_2)} "1";"0"};
{\ar@{->}_<<<{u} "3";"2"};
{\ar@/^1.pc/^{T\varphi(x_1)} "2";"1"};
 {\ar@/_1.pc/_{T\varphi(y_1)} "2";"1"};
 {\ar@/^1.pc/^{\varphi(x_1\inv)} "4";"3"};
{\ar@/_1.pc/_{\varphi(y_1\inv)} "4";"3"};
{\ar@{->}_{\varphi(x_2\inv)} "5";"4"};
{\ar@{=>}^<<<<{\scriptstyle T\varphi(\gamma_1)} (-20,3)*{};(-20,-3)*{}}; 
{\ar@{=>}^<<<<{\scriptstyle \varphi(\gamma_1\inv)} (19,3)*{};(19,-3)*{}}; 
\endxy
\]
and 
\[
\Phi(\gamma_2)(\Phi(y_1) u) \quad =\qquad
\xy
(-50,0)*+{TK}="0";
(-24,0)*+{TK}="1";
(-5,0)*+{\,\,TK}="2";
( 10,0)*+{K}="3";
(29,0)*+{K}="4";
(55,0)*+{K}="5";
{\ar@{->}_{T\varphi (y_1)} "2";"1"};
{\ar@{->}_<<<{u} "3";"2"};
{\ar@/_1.pc/_{T\varphi(x_2)} "1";"0"};
 {\ar@/^1.pc/^{T\varphi(y_2)} "1";"0"};
 {\ar@/^1.pc/^{\varphi(x_2\inv)} "5";"4"};
{\ar@/_1.pc/_{\varphi(y_2\inv)} "5";"4"};
{\ar@{->}_{\varphi(y_1\inv)} "4";"3"};
{\ar@{=>}^<<<<{\scriptstyle T\varphi(\gamma_2)} (-40,3)*{};(-40,-3)*{}}; 
{\ar@{=>}^<<<<{\scriptstyle \varphi(\gamma_2\inv)} (40,3)*{};(40,-3)*{}}; 
\endxy,
\]
\begin{eqnarray*}
\left(\Phi(\gamma_2)(\Phi(y_1) u)   \right) \circ_{vert}
 \left(\Phi(x_2)(\Phi(\gamma_1) u)   \right) 
&=&
\left( T\varphi(\gamma_2)\circ_{hor} T\varphi(\gamma_1)\right)u
\left(\varphi(\gamma_1\inv)\circ_{hor} \varphi(\gamma_2\inv)\right)
\\ 
&=& (T\varphi)(\gamma_2\cdot \gamma_1)\, u\, 
\varphi((\gamma_2\cdot \gamma_1)\inv)\\ 
&=&  \Phi(\gamma_2\cdot \gamma_1)u.
\end{eqnarray*}
We conclude that $\Phi(\gamma_2\cdot \gamma_1) = \Phi(\gamma_2)\circ
_{hor} \Phi(\gamma_1)$ for all arrows $\gamma_2, \gamma_1$ of the Lie
2-group $G$.
It now follows that $\Phi: G\to \mathrm{GL}(\X(K))$ is a homomorphism
of 2-groups.
\end{proof}
\mbox{}

We are now in position to construct the representation $\lambda:G\to
\mathrm{GL}(\X(G))$ of a Lie 2-group $G$ on its 2-vector space $\X(G)$
of vector fields coming from the multiplication on the left. 

\begin{lemma}\label{lem:3.13}
Left multiplication on a Lie 2-group $G$ induces a homomorphism of 2-groups
\[
\lambda: G\to \mathrm{GL}(\X(G))
\]
from $G$ to the 2-group $\mathrm{GL}(\X(G))$ of
automorphisms of the 2-vector space of vector fields on the Lie groupoid $G$.
For each object $x$ of $G$, $\lambda(x):\X(G)\to \X(G)$ is a linear
functor with
\[
\lambda(x)\,(v\stackrel{\alpha}{\Rightarrow}w) =
\xy
(-20,0)*+{TG}="1";
(-5,0)*+{TG}="2";
( 10,0)*+{G}="3";
(25,0)*+{G}="4";
{\ar@{->}_{TL_x} "2";"1"};
{\ar@{->}_{L_{x\inv}} "4";"3"};
{\ar@/^1.pc/^{w} "3";"2"};
{\ar@/_1.pc/_{v} "3";"2"};
{\ar@{=>}^<<<{\scriptstyle \alpha} (3,3)*{};(3,-3)*{}} 
\endxy
\]
for each arrow $v\stackrel{\alpha}{\Rightarrow}w$ of $G$.  Here as
before $L:G\to \Aut(G)$ is the homomorphism of 2-groups induced by
multiplication on the left.  For each arrow $x\xrightarrow{\gamma}y$
of $G$, $\lambda(\gamma):\lambda(x)\to \lambda(y)$ is a natural
isomorphism with
\[
\lambda(\gamma)v =
\xy
(-25,0)*+{TG}="1";
(-5,0)*+{\,\,TG}="2";
( 10,0)*+{G}="3";
(30,0)*+{G}="4";
{\ar@{->}_{v} "3";"2"};
{\ar@/^1.pc/^{TL_x} "2";"1"};
{\ar@/_1.pc/_{TL_y} "2";"1"};
{\ar@/^1.pc/^{L_{x\inv}} "4";"3"};
{\ar@/_1.pc/_{L_{y\inv}} "4";"3"};
{\ar@{=>}^<<<<{\scriptstyle TL_\gamma} (-15,3)*{};(-15,-3)*{}}; 
{\ar@{=>}^<<<<{\scriptstyle L_{\gamma\inv}} (19,3)*{};(19,-3)*{}}; 
\endxy
\]
for all objects $v$ of the 2-vector space $\X(G)$.
\end{lemma}

\begin{proof}
by Corollary~\ref{lem:2.2.0} multiplication on $G$ gives rise to a
homomorphism of 2-groups $L:G\to \Aut(G)$.  By Lemma~\ref{lem:2.2.1}
the homomorphism $L$ gives rise to the homomorphism $\lambda: G\to
\mathrm{GL}(\X(G))$.
\end{proof}

\section{A map of Lie 2-algebras $p: \fg \to \X(G)$} \label{sec:p}

Recall that to a Lie 2-group $G= \{G_1\toto G_0\}$ one can associate a
Lie 2-algebra $\fg= \{\fg_1\toto \fg_0\}$ by applying the $\Lie$
functor to the Lie group $G_0$ of objects, the Lie group $G_1$ of
morphisms and to the structure maps of $G$.  That is, $\fg_0 =
T_{e_0}G_0$, $\fg_1 = T_{e_1}G_1$ and so on.  In this section we
prove:

\begin{theorem} \label{thm:4.1} 
Let $G$ be a Lie 2-group and $\fg$ the associated Lie 2-algebra
There is a morphism of Lie 2-algebras
  $p:\fg\to \X(G)$ from the Lie 2-algebra $\fg$ to the Lie 2-algebra
  of multiplicative vector fields $\X(G)$.  Moreover the functor $p$
  is injective on objects and fully faithful.
\end{theorem}

The theorem has an immediate corollary:
\begin{corollary}
  The image $p(\fg)$ of the functor $p:\fg \to \X(G)$ is a full Lie
  2-subalgebra of the Lie 2-algebra of vector fields $\X(G)$.  This
  2-subalgebra $p(\fg)$ is isomorphic to $\fg$.
\end{corollary}

We construct $p$ as a composite of two of functors.

\subsection{A  Lie 2-algebra $\scrL(G)$ associated to a Lie 2-group $G$}\mbox{}\\

Recall that to define a Lie bracket on the tangent space at the
identity $T_eH$ of a Lie group $H$ one identifies  $T_eH$ with the space of
left-invariant vector fields $\calX(H)^H$ on $H$.  

Similarly given a
Lie 2-group $G$ we define the Lie 2-algebra $\scrL(G)$ 
as follows.
We define the Lie algebra of objects $\scrL(G)_0$ of $\scrL(G)$ to be the
Lie algebra of left-invariant vector fields $\calX(G_0)^{G_0}$ on the
Lie group $G_0$.  We define the Lie algebra of morphisms $\scrL(G)_1$ to
be the Lie algebra $\calX(G_1)^{G_1}$ of left-invariant vector fields
on the Lie group $G_1$.  The source map $\mathbf{s}:\scrL(G)_1 \to
\scrL(G)_0$ is defined by setting the source of a vector field $\alpha
\in \scrL(G)_1$ to be the unique left-invariant vector field $v\in
\scrL(G)_0$ which is $s:G_1\to G_0$ related to $\alpha$.  The target map
$\mathbf{t}:\scrL(G)_1 \to \scrL(G)_0$ is defined similarly.  The unit map
$\mathbf{1}:\scrL(G)_0\to \scrL(G)_1$ is defined by setting $\mathbf{1}_u$
to be the unique left-invariant vector field on $G_1$ which is
$1:G_0\to G_1$ related to $u\in \scrL(G)_0$.  The composition 
\[
\circledast: \scrL(G)_1 \times_{\scrL(G)_0}\scrL(G)_1 \to \scrL(G)_1
\]
is defined pointwise; it is induced by the composition 
\[
\star: TG_1 \times_{TG_0} TG_1\to TG_1
\]
in the tangent groupoid $TG$ (recall that $\star = T\ast$, where
$\ast: G_1\times_{G_0}G_1 \to G_1$ is the composition in $G$).  Thus
$\circledast$ is defined by
\[
(\alpha\circledast\beta)\,(\gamma) := \alpha(\gamma)\star \beta(\gamma)
\]
for all arrows $\gamma\in G_1$.  Routine computations establish that
$\scrL(G)$ is indeed a Lie 2-algebra.

There is an evident functor
\[
\ell: \fg = \{T_{e_1} G_1\toto T_{e_0} G_0\} \to \scrL(G).
\]
which sends a vector $v\in \fg_0 = T_{e_0}G_0$ to the corresponding
left-invariant vector field $\ell(v)$ on the Lie group $G_0$ and an
arrow $\alpha:v\to w \in \fg_1 = T_{e_1}G_1$ to the corresponding
left-invariant vector field $\ell(\alpha)$ on the Lie group $G_1$. By
definition of the Lie brackets on $\fg_0$ and on $ \fg_1$ the maps
$\ell:\fg_0\to \scrL(G)_0$, $\ell: \fg_1\to \scrL(G)_1$ are Lie
algebra maps.  On the other hand $\ell$ is also an isomorphism of
categories --- its inverse is given by evaluation at the
identities: 
\[
\ell\inv (\beta: u\to u') = \beta(e_1): u(e_0)\to
u'(e_0).
\]
\mbox{}

We next
construct a functor $q: \scrL(G)\to \X(G)$.
Given a left-invariant vector field $u$ on the Lie group $G_0$ we
define a multiplicative vector field $q(u)$ as follows.  We take the
object part $q(u)_0:G_0\to TG_0$ to be $u$:
\begin{equation} \label{eq:3.1}
q(u)_0 := u.
\end{equation}
We define $q(u)_1:G_1\to TG_1$ by setting
\begin{equation} \label{eq:3.2}
q(u)_1 (\gamma) = (T\cL_\gamma \circ T1) \,( u(e_0))
\end{equation}
for all $\gamma\in G_1$.  Here and elsewhere in the paper $\cL_\gamma:
G_1\to G_1$ is the left multiplication by $\gamma$ and
$T\cL_\gamma:TG_1\to TG_1$ is its derivative.  It is clear that both
$q(u)_0$ and $q(u)_1$ are vector fields.  It is less clear that
$q(u):G\to TG$ is a functor.  

\begin{lemma}\label{lem:3.1}
  For any vector $u\in T_{e_0} G_0$ the vector field $q(u)_1:G_1\to TG_1$ defined by \eqref{eq:3.2} preserves composition of arrows:
\begin{equation}\label{eq:3.3}
q(u)_1 (\gamma_2 \ast \gamma_1) 
= \left(q(u)_1 (\gamma_2) \right)\star \left(q(u)_1 (\gamma_1)\right).
\end{equation}  
for all composable arrows $(\gamma_2, \gamma_1)\in G_1
\times_{G_0}G_1$. Here as before $\ast$ is the composition in the Lie
groupoid $G$ and $\star $ is the composition in the tangent groupoid
$TG$. 
\end{lemma}

\begin{remark}
  For a manifold $M$ and a point $q\in M$ we write $(q,X)$ for the
  tangent vector $X \in T_qM$.
With this notation it is easy to see that 
\[
T\cL_\gamma (\sigma, X) = Tm ((\gamma, 0), (\sigma, X))
\]
for all $\gamma, \sigma\in G_1$, $X\in T_\sigma G_1$.  Note also that
since the composition 
$\star  = T\ast: T(G_1\times_{G_0}G_1)\to TG_1 $
is fiberwise linear,
\[
 (\gamma_2, 0)\star (\gamma_1, 0) = (\gamma_2 \ast \gamma_1, 0).
\]
\end{remark}

\begin{proof}[Proof of Lemma~\ref{lem:3.1}]   
  Recall that the tangent functor $T:\Man\to \Man$ extends to a
  2-functor $T:\LieGpd\to \LieGpd$ on the 2-category of Lie
  groupoids. As a special case (any Lie group is a Lie groupoid with
  one object) the functor $T$ induces a functor on the category
  $\LieGp$ of Lie groups.  Consequently for a Lie 2-group $G$ its
  tangent groupoid $TG$ is a Lie 2-group as well.  The unit map of the
  groupoid $TG$ is the derivative $T1$ of the unit map $1:G_0\to G_1$.
  The interchange law (see Lemma~\ref{lem:2.3}) in the case of $TG$
  reads:
\begin{equation} \label{eq:3.2}
Tm ((\mu_2\star \mu_1), (\nu_2 \star \nu_1)) =
Tm (\mu_2, \nu_2) \star Tm(\mu_1, \nu_1))
\end{equation}
for all composable pairs $(\mu_2, \mu_1), (\nu_2, \nu_1)\in TG_1
\times _{TG_0}TG_1$.  Now take $\nu_2 = \nu_1 = T1(u(e_0))$ which
we abbreviate as $\bone$.  Then \eqref{eq:3.2} reads:
  \begin{equation}\label{eq:3.3}
Tm ((\mu_2 \star
  \mu_1), (\bone \star \bone))=
(Tm(\mu_2, \bone))\star (Tm(\mu_1, \bone)).
\end{equation}

Then by definition of $q(u)_1$, for any $\gamma\in G_1$
\[
q(u)_1 (\gamma) = T\cL_\gamma \,\bone =
Tm ((\gamma, 0),  \bone).
\] 
Therefore
\begin{eqnarray*}
q(u)_1\,(\gamma_2 \ast \gamma_1) &=& T\cL_{\gamma_2 \ast\gamma_1} (\bone)\\
 &=& Tm ((\gamma_2\ast \gamma_1, 0), \bone)\\
 &=& Tm ((\gamma_2,0)\star (\gamma_1, 0), \bone\star \bone)\\
&=& Tm ((\gamma_2,0), \bone)\star Tm ((\gamma_1, 0), \bone)
\qquad\textrm{ by \eqref{eq:3.3}}\\
&=& T\cL_{\gamma_2}(\bone)\star T\cL_{\gamma_1}(\bone)\\
&=&\left( q({u})_1 (\gamma_2)\right) \star \left(q({u})_1(\gamma_1)\right).
\end{eqnarray*}
\end{proof}
It is easy to see that $Ts\circ q(u)_1 = q(u)_0 \circ s$, $Tt \circ
q(u)_1 = q(u)_0 \circ t$ and $q(u)_1\circ 1 = T1\circ q(u)_0$.  We
conclude that $q(u) = (q(u)_0, q(u)_1):G\to TG$ is a multiplicative
vector field for any left-invariant vector field $u$ on the Lie group
$G_0$.  We thus have constructed the functor $q$ on objects.

An arrow $v\xrightarrow{\alpha} u$ in $\scrL(G)$ is a vector field
$\alpha: G_1\to TG_1$ which is source map $s$ related to $v$ and
target map $t$ related to $u$.  Define $q(\alpha):G_0 \to TG_1$ by
\begin{equation}
q(\alpha) = \alpha \circ 1,
\end{equation}
where as before $1:G_0\to G_1$ is the unit map.  We need to check that
$q(\alpha)$ is an arrow in the category $\X(G)$ from $q(v)$ to $q(u)$.
That is, we need to check that $q(\alpha)$ is a natural transformation
from $q(v)$ to $q(u)$ with
\begin{equation} \label{eq:3.4}
\pi_G (q(\alpha)\,(x)) = 1_x
\end{equation}
for all $x\in G_0$.  

Checking that \eqref{eq:3.4} holds is easy. Since $\alpha$ is a vector
field on $G_1$,
\[
\pi_{G_1} (\alpha (\gamma)) = \gamma
\]
for all $\gamma\in G_1$. In particular $ \pi_{G_1} (\alpha (1_x)) =
1_x$ for all $x\in G_0$, which implies \eqref{eq:3.4}.

We now check that $q(\alpha)$ is in fact a natural transformation from the
functor $q(v)$ to the functor $q(u)$.  Since $\alpha$ is $s$-related
to $v$
\[
Ts (q(\alpha)\,(x)) = Ts (\alpha (1_x)) =v_0 (s(1_x)) = v_0 (x)
\]
for all $x\in G_0$.  Since $q(v)_0 = v_0$ it follows that the source
of the putative natural transformation $q(\alpha):G_0 \to TG_1$ is
$q(v)$.  Similarly the target of $q(\alpha)$ is $q(u)$.  It remains to
check that $q(\alpha)$ is actually a natural transformation: that is, for any
arrow $x\xrightarrow{\gamma} y$ in $G$, the diagram
\[
\xy
(-15,10)*+{q(v)(x)}="1";
(15,10)*+{q(v)(y)}="2";
( -15,-10)*+{q(u)(x)}="3";
(15,-10)*+{q(u)(y)}="4";
{\ar@{->}^{q(v)(\gamma)} "1";"2"};
{\ar@{->}_{q(\alpha)(x)} "1";"3"};
{\ar@{->}^{q(\alpha)(y)} "2";"4"};
{\ar@{->}_{q(u)(\gamma)} "3";"4"};
\endxy
\]
commutes in the category $TG$, i.e.,
\begin{equation} \label{eq:3.5}
q(\alpha)(y)\star q(v)(\gamma) = q(y)(\gamma)\star q(\alpha) (x).
\end{equation}
By definition of  $q(\alpha)$,
\[
q(\alpha) (y)  =\alpha (1_y).
\]
Since $\alpha$ is left-invariant
\[
\alpha (1_y) = T\cL_{1_y}(\alpha (e_1)) = 
   Tm ((1_y, 0), (1_{e_0}, \alpha (1_{e_0}))).
\]
Similarly
\[
q(\alpha) (x) = Tm ((1_x, 0), (1_{e_0}, \alpha (1_{e_0}))).
\]
On the other hand,
\[
q(v)(\gamma) = T\cL_\gamma (\bone _{v_0(e_0)}) = 
Tm ((\gamma, 0), (1_{e_0}, \bone _{v_0(e_0)})).
\]
Similarly,
\[
q(u)(\gamma) = 
Tm ((\gamma, 0), (1_{e_0}, \bone _{u_0(e_0)})).
\]
Now
\begin{eqnarray*}
q(\alpha)(y) \star p(v) (\gamma)
&=& Tm ((1_y,0), (1_{e_0}, \alpha(1_{e_0})))\star
Tm ((\gamma, 0), ((1_{e_0}, v_1(1_{e_0})))\\
&=& Tm ((1_y,0)\star(\gamma, 0), 
(1_{e_0}, \alpha(1_{e_0}))\star (1_{e_0},  \bone_{v_0(e_0)}))\\
&=& Tm ((1_y\ast\gamma, 0), (1_{e_0}, \alpha(1_{e_0}))) \\
&= & Tm( (\gamma, 0),(1_{e_0}, \alpha(1_{e_0}))).
\end{eqnarray*}
Similarly,
\begin{eqnarray*}
 q(u) (\gamma)\star q(\alpha)(x) 
&=& Tm ((\gamma, 0), (1_{e_0}, \bone_{u(1_{e_0})}))\bullet
Tm ((1_x,0), (1_{e_0}, \alpha(1_{e_0})))
\\
&=& Tm ((\gamma, 0)\star (1_x,0), 
( 1_{e_0}, \bone_{u(1_{e_0})})) \star (1_{e_0}, \alpha(1_{e_0}))\\
&= & Tm( (\gamma, 0),(1_{e_0}, \alpha(1_{e_0}))).
\end{eqnarray*}
It follows that \eqref{eq:3.5} holds.  Hence $q(\alpha)$ is an arrow
in $\X(G)$ from $q(v)$ to $q(u)$.

It is not hard to check that the map $q: \scrL(G) \to \X(G)$
constructed above is in fact a functor.  We need to check that $q$ is
a map of Lie 2-algebras.  For this it suffices to check that $q_0:
\scrL(G)_0 \to \X(G)_0$ and $q_1: \scrL(G)_1 \to \X(G)_1$ are Lie
algebra maps.

Recall that the Lie bracket of two multiplicative vector fields
$X=(X_0, X_1)$ and $Y= (Y_0, Y_1)$ is given by
\[
[X,Y]:= ([X_0, Y_0], [X_1, Y_1]).
\]
It follow from the definition of the functor $q$ on objects that for
any two vector fields $u,v\in \scrL(G)_0$
\begin{itemize}
\item[(i)] $[q(u)_0, q(v)_0] = q([u,v])_0$ and
\item[(ii)] $q([u,v])_1$ is the unique left-invariant vector field on
  $G_1$ which is 1-related to $[u,v]$.  Hence
\[
q([u,v])_1 = [q(u)_1, q(v)_1].
\]
\end{itemize}
We conclude that 
\[
q:\scrL(G)_0 \to \X(G)_0
\]
is a map of Lie algebras.

We next check that $q:\scrL(G)_1 \to \X(G)_1$ is also a map of Lie
algebras. Recall the construction of a Lie algebra structure on the
space $\X(G)_1$ starts with the injective linear map $j: \Gamma(A_G)
\to \calX(K_1)$ that maps the section of the Lie algebroid $A_G\to
G_0$ to the corresponding right-invariant vector field (see
\eqref{eq:j}).  We then embed $\X(G)_1$ into the space of vector
fields $\calX(G_1)$ by the map $J$ (see \eqref{eq:J}) and give
$\X(G)_1$ the induced Lie algebra structure: for $\alpha, \beta\in
\X(G)_1$ their bracket $[\alpha, \beta]$ is the unique element of the
vector space $\X(G)_1$ with
\[
J([\alpha, \beta]) = [J(\alpha), J(\beta)].
\]

\begin{lemma}
(We use the notation developed above.)
For any left-invariant vector field $\alpha \in \scrL(G)\equiv \calX (G_1)^{G_1}$
\[
J(q(\alpha)) = \alpha.
\]
Hence $q:\scrL(G)_1 \to \X(G)_1$ is a Lie algebra map.
\end{lemma}

\begin{proof} Since $G$ is a Lie 2-group, for any $(\sigma, \gamma)\in
  G_1\times_{G_0}G_1$
\[
R_\gamma (\sigma) = \sigma * \gamma = \gamma \cdot (1_{t(\gamma)})\inv\cdot \sigma
\]
by Lemma~\ref{lem:1}.  Therefore for any curve $\sigma (\tau)$ in
$G_1$ lying entirely in a fiber of the source map $s:G_1\to G_0$ with
$\sigma (0) = 1_y $ for some $y\in G_0$
\[
TR\gamma (\dot{\sigma}(0)) = \left.\frac{d}{d\tau}\right|_0 \sigma
(\tau)\ast \gamma = \left.\frac{d}{dt}\right|_0 \gamma \cdot (1_y)\inv
\cdot \sigma (\tau)=
T\cL_{\gamma \cdot (1_y)\inv} (\dot{\sigma}(0)).
\]
It follows that for any section $\zeta \in \Gamma (A_G)$ of the
algebroid, $j(\zeta ) \in \calX (G_1)$ is given by
\begin{equation}\label{eq:3.6}
j(\zeta) \, (\gamma) = T\cL_{\gamma \cdot (1_{t(\gamma)})\inv} (\zeta (t(\gamma))).
\end{equation}
Now, for any arrow $\alpha: u\to v$ in $\scrL(G)_1$, 
\[
q(\alpha) = \alpha \circ 1: q(u)\Rightarrow q(v).
\] 
Hence for any arrow $y\xrightarrow{\gamma}x$ in the Lie groupoid $G$
\begin{eqnarray*}
  J(q(\alpha))\,(\gamma) &=& j(q(\alpha)-\bbm{1}_{q(u)}) + q(u)_1 (\gamma)\\
  &=& T\cL_{\gamma \cdot(1_y)\inv} (\alpha (1_y) -T1 (u(y))) 
  + T\cL_\gamma (q(u)_1 (e_1)) \\
  &&\qquad (\textrm{ by \eqref{eq:3.6} and left-invariance of }q(u)_1)\\
  &=& \alpha (\gamma\cdot (1_y)\inv \cdot 1_y) - 
  (T\cL_{\gamma \cdot(1_y)\inv}\circ  T1 \circ T\cL_y)\, u(e_0) 
+ T\cL_\gamma (T1 (u(e_0))).
\end{eqnarray*}
Now for any $z\in G_0$
\[
(\cL_{\gamma \cdot(1_y)\inv} \circ 1 \circ \cL_y)\, (z) = 
\gamma \cdot (1_y)\inv \cdot 1_{yz}= \gamma \cdot  (1_y)\inv  \cdot 1_y \cdot 1_z,
\]
where the last equality holds since $1:G_0\to G_1$ is a homomorphism.  Hence
\[
\cL_{\gamma \cdot(1_y)\inv} \circ 1 \circ \cL_y = \cL_\gamma \circ 1.
\]
and consequently
\[
 T\cL_{\gamma \cdot(1_y)\inv}\circ  T1 \circ T\cL_y =  T\cL_\gamma \circ T1.
\] 
It follows that 
\[
J(q(\alpha)) \, (\gamma) = \alpha (\gamma)
\]
for all $\gamma \in G_1$ and all $\alpha \in \scrL (G)_1$.  

Now by definition of the bracket on the vector space $\X(G)_1$, for
any $\alpha, \beta \in \scrL(G)_1$ the bracket $[q(\alpha), q(\beta)]$
is the unique element of $\X(G)_1$ such that
\[
J ([q(\alpha), q(\beta)]) = [J(q(\alpha)), J(q(\beta))].
\]
On the other hand
\[
J (q [\alpha, \beta])) = [\alpha, \beta] = [J(q(\alpha)), J(q(\beta))]
\]
as well.  Hence,
\[
q([\alpha, \beta]) = [q(\alpha), q(\beta)]
\]
for all $\alpha,\beta\in \cL(G)_1$.
\end{proof}
We conclude that the functor 
\[
q: \scrL(G) \to \X(G)
\]
from the category $\scrL(G)$ of left-invariant vector fields on the
Lie 2-group $G$ to the category $\X(G)$ of multiplicative vector
fields on the Lie groupoid $G$ is a 1-morphism of Lie 2-algebras.  By
construction $q$ is fully faithful and is injective on objects.  
We now define $p:\fg\to \X(G)$ to be the composite
\[
p := q\circ \ell.
\]
By construction $p$ is fully faithful and injective on objects.

%\newpage

\section{Universal properties of the inclusion $i: p(\fg)
  \hookrightarrow \X(G)$ }
\label{sec:u}

As before $G$ denotes a Lie 2-group and $\X(G)$ the Lie 2-algebra of
multiplicative vector fields on the Lie groupoid $G$.

\begin{lemma} \label{lem:4.1} A multiplicative vector field $u =(u_0,
  u_1):G\to TG$ on a Lie 2-group $G$ satisfies
\begin{equation}
\lambda(\gamma)(u) = 1_u
\end{equation}
for an arrow $\gamma$ of $G$ % (as before $1_u : u\to u$ denotes the
% identity arrow in the category $\X(G)$)
if and only if for all $z\in
G_0$
\begin{equation} \label{eq:4.4}
  u_1(1_z) = T\cL_\gamma(u_1 (\cL_{\gamma\inv}(1_z))).
\end{equation}
As before $\cL_\sigma:G_1\to G_1$ denotes the multiplication on the
left by $\sigma\in G_1$ and $\lambda:G\to \mathrm{GL}(\X(G))$ is the
action of $G$ on its vector fields induced by left multiplication (see
Lemma~\ref{lem:3.13}).
\end{lemma}
\begin{proof}
  The proof is a computation. 

Recall that for an arrow $x\xrightarrow{\gamma}y\in G_1$ the
$u$-component of the natural transformation
$\lambda(\gamma):\lambda(x)\Rightarrow \lambda(y)$ is defined to be
the composite
\[
\xy
(-45,0)*+{TG}="1";
(-15, -0)*+{TG}="2";
( 15, 0)*+{G}="3";
{\ar@/^1.pc/^{TL_y} "2";"1"};
{\ar@/_1.pc/_{TL_x} "2";"1"};
{\ar@/^1.pc/^{uL_{y\inv}} "3";"2"};
{\ar@/_1.pc/_{u L_{x\inv}} "3";"2"};
{\ar@{=>}^<<<{\scriptstyle uL_\gamma} (0, 2.5)*{};(0, -2.5)*{}};
{\ar@{=>}^<<<{\scriptstyle TL_\gamma} (-30, 2.5)*{};(-30, -2.5)*{}};
\endxy
\]
Hence for any object $z\in G_0$
\begin{eqnarray*}
  \left(\lambda(\gamma) (u)\right) (z)
  &=& \left((TL_\gamma)\circ _{hor} (uL_{\gamma\inv})\right) (z) \\
  &=& (TL_\gamma)\left((u_0 \circ L_{y\inv})(z)]\right) 
  \star (TL_x) (u_1 (L_{\gamma\inv} (z))  
\qquad \textrm{ (by Remark~\ref{rmrk:3.7.1}})
\end{eqnarray*}
where $\star$ is the composition in the Lie groupoid $TG$.
For any tangent vector $\dot{a}\in T_a G_0$
\[
TL_\gamma (a,\dot{a}) = Tm ((\gamma, 0), (1_a, T1(\dot{a})))
\]
Hence
\[
(TL_\gamma)\left((u_0 \circ L_{y\inv})(z) \right) = 
Tm ((\gamma, 0), (1_{y\inv z}, T1 \,u_1 (1_{y\inv z}))).
\]
For any tangent vector $\dot{\sigma}\in T_\sigma G_1$
\[
TL_x (\sigma,\dot{\sigma}) = Tm ((1_x, 0), (\sigma, \dot{\sigma})).
\]
Hence
\[
(TL_x) (u_1 (L_{\gamma\inv} (z)) 
= Tm ((1_x, 0), (\gamma\inv 1_z, u_1 \gamma\inv 1_z)).
\]
Recall that since $Tm: TG\times TG \to TG$ is a functor, for any two
pairs of composable arrows $ ((\sigma_2 , \dot{\sigma}_2)$, $(\sigma_1,
\dot{\sigma}_1)), ((\sigma_4, \dot{\sigma}_4), (\sigma_3,
\dot{\sigma}_3)) \in TG_1 \times _{TG_0}TG_1 $
\[
Tm  ((\sigma_2 , \dot{\sigma}_2)\star (\sigma_1,
\dot{\sigma}_1)), ((\sigma_4, \dot{\sigma}_4) \star (\sigma_3,
\dot{\sigma}_3)) =
Tm ((\sigma_2 , \dot{\sigma}_2), (\sigma_4, \dot{\sigma}_4)) \star 
Tm ((\sigma_1, \dot{\sigma}_1)) \star  (\sigma_3,\dot{\sigma}_3))
\]
Hence
\begin{eqnarray*}
\left(\lambda(\gamma) (u)\right) (z)
&=& (TL_\gamma)\left((u_0 \circ L_{y\inv})(z)]\right) 
\star (TL_x) (u_1 (L_{\gamma\inv} (z))\\
&=& Tm ((\gamma, 0), (1_{y\inv z}, T1 \,u_1 (1_{y\inv z}))) \star  
Tm ((1_x, 0), (\sigma, \dot{\sigma}))\\
&=&
Tm ((\gamma, 0)\star (1_x, 0), (1_{y\inv z}, u_1(1_{y\inv z}))\star 
(\gamma\inv 1_x, u_1 (\gamma\inv 1_x)))\\
&=& Tm ((\gamma\ast 1_x, 0), (1_{y\inv z} \ast (\gamma\inv 1_x), 
T1(u_0(y\inv z))\star u_1 (\gamma\inv 1-x)))\\
&=& Tm ((\gamma, 0), (\gamma \inv 1_z, u_1 (\gamma\inv 1_z)).
\end{eqnarray*}
Now, for any $\dot{\sigma}\in T_\sigma G_1$,
\[
T\cL_\gamma (\sigma, \dot{\sigma}) = Tm ((\gamma,0), (\sigma, \dot{\sigma})),
\]
where, as before $\cL_\gamma: G_1\to G_1$ is the left multiplication
by $\gamma$ and $T\cL_\gamma:TG_1\to TG_1$ is its derivative.  Therefore
\[
\left(\lambda(\gamma) (u)\right) (z) = T\cL_\gamma (u_1 (\cL_{\gamma\inv} 1_z)).
\]
Since the $z$ component of $1_u:u\Rightarrow u$ is $T1 (u_0(z))$ and
since $T1 (u_0(z)) = u_1 (1_z)$ (because $u:G\to TG$ is a functor),
the result now follows:
\[
T\cL_\gamma (u_1 (\cL_{\gamma\inv} 1_z)) = u_1 (1_z).
\]
\end{proof}

\begin{theorem} \label{thm:5.2} Let $G$ be a Lie 2-group, $\fg$ the
  associated tangent Lie 2-algebra, $p:\fg\to \X(G)$ is the map of
  2-vector spaces constructed in Theorem~\ref{thm:4.1},
  $i:p(\fg)\hookrightarrow \X(G)$ the inclusion of 2-vector spaces and
  $\lambda: G\to \Aut(\X(G))$ is the action of the Lie 2-group $G$ on
  its Lie 2-algebra of multiplicative vector fields by left
  multiplication (see Lemma~\ref{lem:3.13}).
\begin{enumerate}
\item The diagram
\begin{equation}\label{diag:4.1}
\xy
(0,15)*+{p(\fg)}="1";
(-15, -5)*+{\X(G)}="2";
( 15, -5)*+{\X(G)}="3";
{\ar@{->}_{i} "1"; "2"};
{\ar@{->}^{i} "1"; "3"};
{\ar@/^1.pc/^{\lambda(y)} "3";"2"};
{\ar@/_1.pc/_{\lambda(x)} "3";"2"};
{\ar@{=>}^<<<{\scriptstyle \lambda(\gamma)} (0,-2.5)*{};(0,-7)*{}} 
\endxy
\end{equation}
commutes for any choice of arrow $x\xrightarrow{\gamma} y\in G_1$.
That is,
\[
\lambda(x)\circ i = i
\]
for all $x\in G_0$ and 
\[
\lambda(\gamma)i =1_i
\]
for all $\gamma \in G_1$ (here $\lambda(\gamma)i$ is the whiskering of
the natural transformation $\lambda(\gamma)$ by the functor $i$).
\item For any map $\psi:\fh \to \X(G)$ of 2-vector spaces such that
  the diagram
\begin{equation}\label{diag:4.2}
\xy
(0,15)*+{\fh}="1";
(-15, -5)*+{\X(G)}="2";
( 15, -5)*+{\X(G)}="3";
{\ar@{->}_{\psi} "1"; "2"};
{\ar@{->}^{\psi} "1"; "3"};
{\ar@/^1.pc/^{\lambda(y)} "3";"2"};
{\ar@/_1.pc/_{\lambda(x)} "3";"2"};
{\ar@{=>}^<<<{\scriptstyle \lambda(\gamma)} (0,-2.5)*{};(0,-7)*{}} 
\endxy
\end{equation}
commutes for all choices of arrows $x\xrightarrow{\gamma} y\in G_1$
there exists a unique map of 2-vector  spaces $\bar{\psi}:\fh \to p(\fg)$
so that
\[
\psi = i\circ \bar{\psi}.
\]
\end{enumerate}
In other words $i:p(\fg)\to \X(G)$ is a (strict conical) limit of the
functor $\lambda: G\to \mathrm{GL}(\X(G))$.
\end{theorem}

\begin{remark}
In the course of the proof we realize the limit of the functor $\lambda: G\to \mathrm{GL}(\X(G))$ explicitly as a sub 2-vector space of the 2-vector space $\X(G)$ cut out by equations.
\end{remark}

\begin{proof}
We argue first
\begin{itemize}
\item[(i)]  For any multiplicative vector field $u=(u_0, u_1):G\to TG$
\begin{equation} \label{eq:4.5}
\lambda(x)u = u \quad \textrm{ for all } x\in G_0\qquad \textrm{ and } \qquad
\lambda(\gamma)u = 1_u \quad \textrm{ for all }\gamma \in G_1
\end{equation}
if and only if
\begin{equation}\label{eq:4.6}
u = p(u(e_0)).
\end{equation}

\item[(ii)] For any morphism $\alpha:u\Rightarrow v$ in the category $\X(G)$
\begin{equation} \label{eq:4.5ii}
\lambda (x) (\alpha) = \alpha \textrm{  for all }x\in G_0
\end{equation}
if and only if 
\begin{equation}\label{eq:4.6ii}
\alpha = p(\alpha (e_0)).
\end{equation}
\end{itemize}

\begin{proof}[Proof of (i)] By Lemma~\ref{lem:4.1}
  $\lambda(\gamma))(u) = 1_u$ if and only if \eqref{eq:4.4} holds for
  all $z\in G_0$.  It is easy to see that \eqref{eq:4.4} is equivalent
  to the vector field $u_1$ on the Lie group $G_1$ being
  left-invariant.

On the other hand $\lambda(x)u = u$ for all $x\in G_0$ translates into
\begin{eqnarray*}
T\cL_x \circ u_0 \circ \cL_{x\inv} = u_0\\
\textrm{and}\\
T\cL_{1_x} \circ u_1 \circ \cL_{1_x\inv} = u_1.
\end{eqnarray*}
Thus \eqref{eq:4.5} implies that $u_0$ and $u_1$ are both
left-invariant. Moreover, since $u$ is multiplicative and $e_1 =
1(e_0)$, $u_1 (e_1) = T1 u_0(e_0)$.  Hence \eqref{eq:4.5} implies
\eqref{eq:4.6}.

Conversely, suppose $(u_0, u_1) = p(a)$ for some $a\in \fg_0$.  By
construction of the functor $p$, $a= u_0(e_0)$.  Moreover $u_1$ is a
left-invariant vector field on $G_1$ with $u_1(e_1) = T1 (u_0 (e_0))$.
Hence equation \eqref{eq:4.4} hold for all $z\in G_0$ and all $\gamma \in G_1$,
which implies that $\lambda(\gamma)u = 1_u $ for all $\gamma \in G_1$.
It also implies that 
\begin{equation}
T\cL_{1_x} \circ u_1 \circ \cL_{1_{x\inv}} = u_1
\end{equation}
for all $x\in G_1$.  On the other hand, by construction of the functor
$p$ the vector field $u_0$ on $G_0$ is left-invariant.  Hence
\begin{equation}
T\cL_{x} \circ u_0 \circ \cL_{{x\inv}} = u_0
\end{equation}
for all $x\in G_0$.  Therefore $\lambda (x) u = u$ for all $x\in G_0$.
We conclude that if $u = p(a)$ then \eqref{eq:4.5} holds.  This
finishes our proof of (i).
\end{proof}
\begin{proof}[Proof of (ii)]
By definition of the functor $\lambda(x)$, 
\[
(\lambda(x)\alpha)(z) = TL_x (\alpha (x\inv z))
\]
for all $z\in G_0$.
By definition of the functor $L_x$ on arrows,
\[
 TL_x (\alpha (x\inv z)) = T\cL_{1_z} (\alpha (x\inv z)
\]
where as before $\cL_{1_z}$ is left multiplication by $1_z\in G_1$.
Thus if $\lambda (x)\alpha = \alpha$  then
\[
\alpha(x)  = T\cL_{1_x} (\alpha (e_0))
\]
Hence \eqref{eq:4.5ii} implies \eqref{eq:4.6ii}.

Conversely,  if $\alpha = p(b)$ for some $b\in \fg_1$  then $\alpha(x) = T\cL_{1_x} b$ for all $x\in G_0$ and $\alpha (e_0)= T\cL_{1_{e_0}} b = b$.  This
finishes our proof of (ii).
\end{proof}
The proof of part (1) of the theorem is now easy.  By (i), for any
object $a\in \fg_0$, and any object $x\in G_0$,
\[
\lambda (x)(p(a)) = p(a).
\]
By (ii), for any object $b\in \fg_1$ and any arrow $\gamma \in G_1$
\[
\lambda(\gamma)(p(b)) = p(b).
\]
Hence \eqref{diag:4.1} commutes.

Now suppose $\psi: \fh\to \X(G)$ is a map of 2-vector spaces making the
diagram \eqref{diag:4.2} commute.
Then for any object $X$ of $\fh$
\[ %begin{equation} \label{eq:4.5}
\lambda(x) \psi(X)= \psi(X) \quad \textrm{ for all } x\in G_0\qquad 
\textrm{ and } \qquad
\lambda(\gamma)\psi(X) = 1_{\psi(X)} \quad \textrm{ for all }\gamma \in G_1
\] %end{equation}
Consequently by (i)
\[
\psi(X) = p (\psi(X) (e_0)).
\]
The commutativity of  \eqref{diag:4.2} also implies that 
\[ 
\lambda (x) (\psi(Y)) = \psi(Y) \textrm{  for all }x\in G_0
\]
for any arrow $Y$ in $\fh$.  Then by (ii)
\[ 
\psi(Y) = p(\psi(Y)\, (e_0)).
\] 
We conclude that the image of $\psi:\fh\to \X(G)$ is contained in $p(\fg)$ 
and the result follows.
\end{proof}
We are now in position to prove our main result by putting together
all the work we have already done.

\begin{proof}[Proof of Theorem~\ref{thm:main}]
  By Lemma~\ref{lem:3.13} the action of the Lie 2-group $G$ on itself
  by multiplication on the left gives rise to a homomorphism of
  2-groups $\lambda:G\to \mathrm{GL}(\X(G))$.  By
  Theorem~\ref{thm:4.1} we have a 1-morphism of Lie 2-algebras
  $p:\fg\to \X(G)$ which is fully faithful and injective on objects.
  In particular $p:\fg\to p(\fg)$ is an isomorphism of Lie 2-algebras.  

  On the other hand by Theorem~\ref{thm:5.2}, the 2-vector space
  $p(\fg)$ underlying the Lie 2-algebra $p(\fg)$ is a limit of the
  functor $\lambda:G\to \mathrm{GL}(\X(G))$.  Hence it makes sense to
  say that $p(\fg)$ is the 2-vector space $\X(G)^G$ of left-invariant
  vector fields.  As we remarked previously $p(\fg)$ is also a Lie
  2-subalgebra of $\X(G)$ which is isomorphic to the Lie 2-algebra
  $\fg$.
\end{proof}

\end{document}